\documentclass[12pt]{amsart}
\usepackage{lmodern}
\usepackage[initials]{amsrefs}
\usepackage[english]{babel}
\usepackage{amsmath}
\usepackage{amsthm}
\usepackage{mathtools}
\usepackage{mathrsfs}
\usepackage{amssymb}
\usepackage{xcolor}
\usepackage{xfrac}
\usepackage{esint}
\usepackage{graphicx}
\usepackage{float}
\usepackage{array}
\usepackage{enumitem}
\usepackage[new]{old-arrows}
\usepackage[all,cmtip]{xy}
\usepackage{tikz-cd}
\usepackage{centernot}
\usepackage{stmaryrd}
\usepackage{comment}
\excludecomment{confidential}

\newcommand{\co}{\mathfrak c}

\usepackage[margin=2in]{geometry}

\newcommand{\EE}{\mathcal E}
\newcommand{\EEE}{\mathbb E}

\numberwithin{equation}{section}

\newtheorem{theorem}{Theorem}[section]
\newtheorem{corollary}[theorem]{Corollary}
\newtheorem{lemma}[theorem]{Lemma}
\newtheorem{proposition}[theorem]{Proposition}
\newtheorem{definition}[theorem]{Definition}

\newtheorem{conjecture}[theorem]{Conjecture}

\DeclareMathOperator{\supp}{supp}

\DeclareMathOperator{\spann}{span}

\newcommand{\N}{\mathbb{N}}
\newcommand{\Z}{\mathbb{Z}}
\newcommand{\Q}{\mathbb{Q}}
\newcommand{\R}{\mathbb{R}}

\def\a{\alpha}

\def\d{\delta}

\def\G{\Gamma}
\def\l{\lambda}

\def\o{\omega}

\def\s{\sigma}

\def\z{\zeta}

\newcommand{\mc}{\mathcal}
\newcommand{\mf}{\mathfrak}

\DeclareMathOperator{\vol}{\mathbf{vol}}

\begin{document}
	\title[Isometric classification of the $L^{p}$-spaces of infinite Lebesgue measure]{Isometric classification of the $L^{p}$-spaces of infinite dimensional Lebesgue measure} 
	
	\author[D. L. Rodríguez-Vidanes]{Daniel L. Rodríguez-Vidanes}\thanks{}
	\address{Departamento de Matemática Aplicada a la Ingeniería Industrial \\
		E.T.S.I.D.I. \\
		Ronda de Valencia 3 \\
		Universidad Politécnica de Madrid \\
		Madrid, 28012, Spain.}
	\email{dl.rodriguez.vidanes@upm.es}
	
	\author[J. C. Sampedro]{Juan Carlos Sampedro} \thanks{The second author has been supported by the Research Grant PID2024--155890NB-I00 of the Spanish Ministry of Science and Innovation and by the Institute of Interdisciplinar Mathematics of Complutense University.}
	\address{Departamento de Matemática Aplicada a la Ingeniería Industrial \\
		E.T.S.I.D.I. \\
		Ronda de Valencia 3 \\
		Universidad Politécnica de Madrid \\
		Madrid, 28012, Spain.}
	\email{juancarlos.sampedro@upm.es}

\begin{abstract}
	We investigate the isometric structure of $L^{p}$-spaces for the infinite-dimensional Lebesgue measure $(\mathbb{R}^{\mathbb{N}},\mu)$. Under the continuum hypothesis (CH) we prove $L^{p}(\mu)\cong \ell^{p}(\mathfrak{c},L^{p}[0,1])$, where $\mf{c}$ denotes the cardinality of the continuum, and without CH we obtain an isometric, complemented copy of $\ell^{p}(\mathfrak{c},L^{p}[0,1])$ inside $L^{p}(\mu)$. In a general framework, we characterize precisely when $L^{p}(\nu)\cong \ell^{p}(\kappa,L^{p}[0,1])$ and classify all such isometries.
\end{abstract}

\keywords{$L^{p}$-spaces, Continuum Hypothesis, Lebesgue measure, Isometric Classification}
\subjclass[2010]{46B04 (primary), 28C20, 46G12 (secondary)}

\maketitle

\tableofcontents

\section{Introduction}

The isometric classification of $L^p$-spaces has been a central topic of research throughout the 20th and 21st centuries (see, for instance, \cite{EnSt,EnRo,La,HaRoSu,Fremlin,Fremlin2,Fremlin3,Fremlin4,Fremlin51,Fremlin52,LiPe,M,Pa,Ro,Ro2,Ro3,SP4}, among many others). A fundamental consequence of Maharam's theorem (see, for example, \cite[Th. 14, p. 130]{La}, \cite[Th. 332B]{Fremlin3}, or the original work \cite{M}) states that, for any measure space $(X, \Sigma, \nu)$, there exist two sets $\G, I$ and two families of cardinals $\tau_i$, $\kappa_i$, $i \in I$, such that the following isometric isomorphism holds:

\begin{equation} \label{Eqq}
	L^{p}(\nu)\cong\ell^{p}(\G)\oplus_{p}\left(\bigoplus_{i\in I}\ell^{p}(\tau_{i},L^{p}([0,1]^{\kappa_{i}}))\right)_{p}.
\end{equation}

Throughout this article, given a family $(X_{j})_{j\in J}$ of Banach spaces, we define the space
\[
\left(\bigoplus_{j\in J}X_{j}\right)_{p} := \left\{(x_{j})_{j\in J}\in \bigtimes_{j\in J}X_{j} \; : \; \sum_{j\in J} \|x_{j}\|_{X_{j}}^{p}<\infty\right\},
\]
with norm
\[
\|(x_{j})_{j\in J}\|_{\ell^{p}}:=\left(\sum_{j\in J} \|x_{j}\|_{X_{j}}^{p}\right)^{\frac{1}{p}}.
\]
In the case where $X_{j} = X$ for every $j \in J$, we denote $\ell^{p}(J,X) := \left(\bigoplus_{j\in J}X\right)_{p}$.

In this paper, we study the isometric structure of the $L^p$-spaces associated with an infinite-dimensional analogue of the Lebesgue measure constructed in \cite{RB2} and generalized in \cite{SP3}. To be more precise, we deal with the measure space $(\mathbb{R}^{\mathbb{N}}, \mathcal{B}_{\infty}, \mu)$, where $\mathcal{B}_{\infty}$ is the $\s$-algebra generated by the \textit{cylinder sets}
\begin{equation*}
	\left\{\bigtimes_{i=1}^{m}C_{i} \times \bigtimes_{i=m+1}^{\infty} \mathbb{R} \; : \; C_{i} \in \mathcal{B}, \ \forall i\in\{1,2,...,m\}, \ m\in\mathbb{N}\right\},
\end{equation*}
and $\mc{B}$ denotes the Borel $\s$-algebra of $\R$.
It is worth noting that $\mathcal{B}_{\infty}$ coincides with the Borel $\sigma$-algebra of $\mathbb{R}^{\mathbb{N}}$ when endowed with the product topology.

For completeness, we sketch the construction of the measure $\mu$.
By $\mathcal{F}(\mc{B},\l)$ we denote the set of \textit{finite rectangles} on $\R^{\N}$ defined by
\begin{equation*}
	\mathcal{F}(\mc{B},\l):=\left\{\bigtimes_{i\in\mathbb{N}}C_{i} \; : \; C_{i}\in\mc{B},\; \l(C_i)<\infty, \; \forall i\in \N, \; \prod_{i\in\mathbb{N}}\l(C_{i})\in[0,\infty) \right\},
\end{equation*}
where $\l$ denotes the standard Lebesgue measure on $(\R,\mc{B})$; and $\vol$ for the \textit{volume} map $\vol: \mathcal{F}(\mc{B},\l)\to[0,\infty)$ given by 
\begin{equation*}
	\vol\left(\bigtimes_{i\in\mathbb{N}}C_{i}\right):=\prod_{i\in\mathbb{N}}\l(C_{i}).
\end{equation*}
We denote by $\mathcal P(\R^{\N})$ the set of all subsets of $\R^{\N}$.
The measure $\mu$ is defined as the restriction to $\mc{B}_{\infty}$ of the outer measure $\mu^{\ast}:\mathcal{P}(\R^{\N})\to [0,\infty]$ defined, for every $A\in\mathcal{P}(\R^{\N})$, by
\begin{equation}
	\label{JCS}
\mu^{\ast}(A):=\text{inf}\left\{ \sum_{n\in\mathbb{N}}\vol(\mathscr{C}_{n}) : \{\mathscr{C}_{n}\}_{n\in\mathbb{N}}\subset\mathcal{F}(\mc{B},\l)\text{ and } A\subset\bigcup_{n\in\mathbb{N}}\mathscr{C}_{n} \right\},
\end{equation}
where we set $\inf \varnothing =\infty$. Therefore, $\mu$ is a Borel translation-invariant measure on $\R^{\N}$ satisfying the identity
$$
\mu\left(\bigtimes_{i\in\N}C_{i}\right)=\prod_{i\in\N}\l(C_{i}),
$$
for each $\bigtimes_{i\in\N}C_{i}\in \mathcal{F}(\mc{B},\l)$.

The study of translation-invariant measures on $\mathbb{R}^{\mathbb{N}}$ has been a subject of interest in the literature. In \cite{RB}, Baker constructed for the first time a translation-invariant Borel measure $\mu_B$ on $(\mathbb{R}^{\mathbb{N}}, \mathcal{B}_{\infty})$ that behaves as an infinite-dimensional analogue of the Lebesgue measure, using rectangles of the form $\bigtimes_{i\in\mathbb{N}}(a_i,b_i)$. Later, in \cite{RB2}, Baker introduced the measure $\mu$ as described above. (In \cite{RB2}, Baker considered the completion of $\mu$ to obtain an infinite-dimensional analogue of the Fubini--Tonelli theorem.) In that work, Baker raised the question of whether the measure spaces corresponding to $\mu_B$ and $\mu$ coincide. However, in \cite{Pan}, Pantsulaia showed in Remark~2.2 that while $\mu_B$ is absolutely continuous with respect to $\mu$, it is not equivalent to it. In fact, a simple observation shows that $\mu(\{0,2\}^{\mathbb{N}}) = 0$ while $\mu_B(\{0,2\}^{\mathbb{N}}) = \infty$.

It is well known that infinite-dimensional locally convex topological vector spaces do not admit a nontrivial translation-invariant $\sigma$-finite Borel measure \cite[p. 143]{Y}. Consequently, $\mu$ is not $\sigma$-finite.  Moreover, it is known that $\mu_B$ is not localizable in the sense of Fremlin \cite[Def. 211G]{Fremlin2}, since it is not semi-finite (see \cite[Remark~5.2]{Pan}). In Appendix \ref{AP}, we prove that $\mu$ is neither localizable nor semi-finite.

For the measure space \( (\R^{\N},\mu) \), determining the family of cardinals involved in the isometric isomorphism \eqref{Eqq} remains an open problem. The first main result of this paper solves this problem under the validity of the continuum hypothesis (CH).

\begin{theorem}
	\label{Th8.1.2}
	Let \( 1\leq p <\infty \). Then, under CH, the following isometric isomorphism holds:  
	\begin{equation}
		\label{Eq}
	L^{p}(\mu)\cong \ell^{p}(\mf{c},L^{p}[0,1]).
	\end{equation}
\end{theorem}  

In the general case in which CH is not guaranteed, we prove that, $L^{p}(\mu)$ admits at least an isometric complemented copy of $\ell^{p}(\mf{c},L^{p}[0,1])$ via a cube decomposition of $\R^{\N}$. We believe (and therefore, we conjecture) that the isometric isomorphism \eqref{Eq} is independent of ZFC. More precisely, we conjecture that there is a model of ZFC in which the isometric isomorphism \eqref{Eq} does not hold. (See the note added in Section~\ref{Se4} for its subsequent resolution.)

Throughout the article we use standard notation and terminology from set theory \cite{Ci,Jech} and measure theory \cite{C,F}, unless stated otherwise. For latter considerations, we recall that in a measure space \((X,\Sigma,\nu)\), a measurable set \(A\in\Sigma\) with \(\nu(A)>0\) is an \emph{atom} if every measurable \(B\subset A\) satisfies \(\nu(B)\in\{0,\nu(A)\}\).
The space \((X,\Sigma,\nu)\) is \emph{purely nonatomic} if it has no atoms (equivalently, no measurable subset of positive measure is an atom).
We say that \((X,\Sigma,\nu)\) is \emph{relatively nonatomic} if for every \(\sigma\)-finite measurable subset \(E\in\Sigma\), the restricted measure space \((E,\Sigma|_{E}, \nu)\) is purely nonatomic, where $\Sigma|_{E}$ is the restriction of $\Sigma$ to $E$, i.e., $\Sigma|_{E}:=\{B\cap E : B\in \Sigma\}$; equivalently, no \(\sigma\)-finite measurable subset of positive measure is an atom.
The latter notion is new and is introduced in this article.

\par The paper is organized as follows. In Section \ref{Se2}, we prove Theorem \ref{Th8.1.2} by carefully choosing an appropriate partition of $\R^{\N}$ whose existence strongly depends on CH. Section \ref{Se3} is devoted to study the general case, that is, when CH is not guaranteed. We prove that $L^{p}(\mu)$ contains an isometric and complemented copy of $\ell^{p}(\mf{c},L^{p}[0,1])$. This analysis is performed via a cube decomposition of $\R^{\N}$ that splits $L^{p}(\mu)$ into two complemented subspaces $G^{p}$ and $B^{p}$. We prove that the space $G^{p}$ is isometrically isomorphic to $\ell^{p}(\mf{c},L^{p}[0,1])$ while $B^{p}$ contains an isometric and complemented copy of $\ell^{p}(\mf{c},L^{p}[0,1])$. The final Section \ref{Se4} characterizes when $L^{p}(\nu)$, for some relatively nonatomic measure $\nu$, is isometrically isomorphic to $\ell^{p}(\kappa, L^{p}[0,1])$ for some cardinal number $\kappa$.
Moreover, under natural hypotheses, when these spaces are isometrically isomorphic we describe all isometric isomorphisms between them. We conclude by applying these results to conjecture that the isometric isomorphism $L^{p}(\mu)\cong \ell^{p}(\co,L^{p}[0,1])$ is independent of ZFC. Finally, we include Appendix A in which we prove some properties and aspects of the measure $\mu$ not present in the literature.

\section{Isometry under the Continuum Hypothesis}\label{Se2}

In this section we prove that, under the Continuum Hypothesis, CH, $L^{p}(\mu)$ is isometrically isomorphic to $\ell^{p}(\mf{c},L^{p}[0,1])$.  Therefore, throughout this section we work under the validity of CH. 
\par Under this assumption, since $|\mc{F}(\mc{B},\l)|=\mf{c}$, we can consider 
an enumeration $\left\{R_{\xi} : \xi < \o_{1}\right\}$ of $\mc{F}(\mc{B},\l)$. For each $\xi < \o_{1}$, we denote
\begin{equation}
	\label{Eq2}
	B_{\xi}:=R_{\xi} \setminus \bigcup_{\z<\xi} R_{\z}.
\end{equation}
Note that each $B_{\xi}$ is Borel as $\bigcup_{\z<\xi}R_{\z}$ is a countable union and $\mu(B_{\xi})<\infty$ as $B_{\xi}\subset R_{\xi}$. Moreover, for distinct $\xi_{1},\xi_{2}<\o_{1}$, $B_{\xi_{1}}\cap B_{\xi_{2}}=\varnothing$. Therefore,
$$
\R^{\N}= \bigsqcup_{\xi<\o_{1}}B_{\xi}.
$$
Throughout this article, the notation $ \sqcup$ stands for the disjoint union.
\par The first result of this section yields a partition of $\R^{\N}$ into Borel sets of finite measure with a property needed for the proof of the section’s main result.

\begin{proposition}
	\label{Pr2.1}
	There exists a family $\mathcal F=\{C_\xi\colon \xi<\o_{1}\}\subset\mc{B}_{\infty}$ of pairwise disjoint Borel sets of finite measure, possibly empty, covering $\R^\N$ such that for every Borel $B\in\mc{B}_{\infty}$ of finite measure
	the set $\mathcal F(B):=\{\xi<\o_{1}\colon B\cap C_\xi\neq \varnothing\}$ is countable. 
\end{proposition}

\begin{proof}
	Set
	\begin{equation}
		\label{e2.2}
		C_{\xi}:=B_{\xi},\qquad \xi<\omega_1.
	\end{equation}
	By \eqref{Eq2}, these sets are Borel, pairwise disjoint, contained in $R_\xi$, and cover $\R^{\N}$.
	
	Take $B\in\mc{B}_{\infty}$ with finite measure. Then, by the definition of the measure $\mu$, there exists a countable subset $\{R_{\xi_n}\}_{n<\o}\subset \mc{F}(\mc{B},\l)$ such that
	$$
	B\subset \bigcup_{n<\o}R_{\xi_n}\subset\bigcup_{n<\o}\bigcup_{\eta\leq\xi_n} B_\eta.
	$$
	Since $\{\xi<\o_{1} : B\cap B_\xi\neq \varnothing\}$ is countable, necessarily $\mathcal F(B)=\{\xi<\o_{1}\colon B\cap C_\xi\neq \varnothing\}$ is also countable. The proof is concluded.
\end{proof}

The next result is the key to prove the required isometric isomorphism.

\begin{proposition}
	\label{Le5.1}
	Let $1\leq p<\infty$. Then, for every $f\in L^{p}(\mu)$, it holds that
	$$
	\int_{\R^{\N}}|f|^{p} \; d\mu = \sum_{\xi<\o_{1}}\int_{C_{\xi}}|f|^{p}\; d\mu,
	$$
	where $\{C_\xi\colon \xi<\o_{1}\}$ are the sets introduced in Proposition \ref{Pr2.1}.
\end{proposition}

\begin{proof}
	Let $f\in L^{p}(\mu)$. Then $\supp(f)$ is $\sigma$-finite. Indeed, 
	$$
	\supp(f)=\bigcup_{n\in\N}[|f|> 1/n],
	$$
	where
	$$
	[|f|> 1/n]:=\left\{x\in\R^{\N} : |f(x)|>1/n\right\},
	$$
	and, by Chebyshev's inequality,
	$$
	\mu\left([|f|^{p}>1/n]\right)\leq n\int_{\R^{\N}}|f|^{p}\; d\mu = n \|f\|^{p}_{L^{p}}<\infty.
	$$
	This implies that there exists an $\alpha_{n}<\omega_1$ such that 
	$$
	[|f|^{p}>1/n]\subset \bigcup_{\xi<\alpha_{n}} R_\xi.
	$$
	Consequently, 
	$$
	\supp(f)=\bigcup_{n\in\N}[|f|^{p}>1/n]\subset \bigcup_{n\in\N}\bigcup_{\xi<\alpha_{n}} R_\xi \subset \bigcup_{\xi <\a} R_{\xi},
	$$
	where $\a:=\sup_{n}\a_{n}<\o_{1}$ since the cofinality of $\omega_1$ is $\omega_1$. 
	
	By \eqref{e2.2}, $C_\xi=B_\xi$ for every $\xi<\omega_1$. Set $\beta=\alpha$. Therefore,
	$$
	\supp(f)\subset \bigcup_{\xi<\alpha} R_\xi= \bigsqcup_{\xi<\alpha} B_\xi =  \bigsqcup_{\xi<\beta}C_{\xi}.
	$$
	Hence, we deduce that
	$$
	\int_{\mathbb{R}^{\mathbb{N}}} |f|^{p} \ d\mu
	=\int_{ \bigsqcup_{\xi<\beta} C_\xi} |f|^{p} \ d\mu
	=\sum_{\xi<\beta} \int_{C_{\xi}} |f|^{p} \ d\mu
	=\sum_{\xi<\omega_1} \int_{C_{\xi}} |f|^{p} \ d\mu.
	$$
	The proof is concluded.
\end{proof}

Before presenting the main result of this section, we need the following lemma.

\begin{lemma}
	\label{Le1}
	Let $M\in\mc{B}_{\infty}$ be such that $\mu(M)<\infty$. Then, $L^{p}(M,\mu)$ is separable for every $1\leq p <\infty$.
\end{lemma}

\begin{proof}
	Let $M\in\mc{B}_{\infty}$ with $\mu(M)<\infty$. By the definition of the Borel $\s$-algebra $\mc{B}_{\infty}$, it is apparent that $\mc{B}_{\infty}$ is the $\s$-algebra generated by
	$$
	\mc{Q}:=\left\{\bigtimes_{i=1}^{m}B_{i}\times \bigtimes_{i=m+1}^{\infty}\R \; : \; B_{i}\in\mc{B}_{r}, \; \forall i\in \{1,2,\dots,m\} \; \text{and} \; m\in\N\right\}
	$$
	with $\mc{B}_{r}:=\{(a,b)  :  a,b \in\Q, \; a<b\}$. In particular, $\mc{B}_{\infty}|_{M}:=\{B\cap M : B\in\mc{B}_{\infty}\}$ is a countably generated $\s$-algebra and $\mu|_{M}$ is a finite measure. The proof concludes by noting that the $L^{p}$-space of a finite measure space with a countably generated $\s$-algebra, is separable for every $1\leq p <\infty$ (see, for instance, Proposition 3.4.5 of Cohn \cite{C}).
\end{proof}

We proceed to state and prove the main result of this section.

\begin{theorem}
	\label{TeoPri}
	Let $1\leq p <\infty$. Then, the space $L^{p}(\mu)$ is isometrically isomorphic to $\ell^{p}(\co,L^{p}[0,1])$. Symbolically,
	$$
	L^{p}(\mu)\cong \ell^{p}(\co,L^{p}[0,1]).
	$$
\end{theorem}

\begin{proof}
	Let $\{C_\xi\colon \xi<\o_{1}\}$ be the sets introduced in Proposition \ref{Pr2.1} and consider the linear operator
	\begin{equation*}
		T: L^{p}(\mu)  \longrightarrow \left(\bigoplus_{\xi < \o_{1}} L^{p}(C_{\xi})\right)_{p}, \quad  T(f) := (f_{\xi})_{\xi<\o_{1}},
	\end{equation*}
	where $f_{\xi}:= f  \upharpoonleft C_{\xi}$ for every $\xi<\o_{1}$.
	The operator $T$ is well defined as a consequence of Proposition \ref{Le5.1}. Indeed, note that
	\begin{align*}
		\|Tf\|_{\ell^{p}}^{p}=\sum_{\xi<\o_{1}}\|f_{\xi}\|_{L^{p}(C_{\xi})}^{p}= \sum_{\xi<\o_{1}}\int_{C_{\xi}}|f|^{p} \; d\mu  = \int_{\R^{\N}}|f|^{p} \; d\mu < \infty.
	\end{align*}
	Moreover, this proves that $T$ is an isometry. Now, let 
	$$
	\bar{f}=(f_{\xi})_{\xi<\o_{1}}\in \left(\bigoplus_{\xi < \o_{1}} L^{p}(C_{\xi})\right)_{p}.
	$$ 
	Note that $\supp(\bar{f})$ is countable, where $\supp(\bar{f}):=\{\xi < \o_{1} : \|f_{\xi}\|_{L^{p}} \neq 0\}$. 
	Define $f:\R^{\N}\to\R$ as
	$$
	f:=\sum_{\xi\in\supp(\bar{f})}f_{\xi} \; \chi_{C_{\xi}}.
	$$
	Then, it follows that $f\in L^{p}(\mu)$ as
	$$
	\|f\|^{p}_{L^{p}}=\sum_{\xi\in \supp(\bar{f})}\int_{C_{\xi}}|f_{\xi}|^{p}\; d\mu = \sum_{\xi\in \supp(\bar{f})} \|f_{\xi}\|^{p}_{L^{p}}= \sum_{\xi<\o_{1}} \|f_{\xi}\|^{p}_{L^{p}}=\|\bar{f}\|_{\ell^{p}}^{p}<\infty.
	$$
	Moreover, $T(f)=\bar{f}$. Therefore, $T$ is an isometric isomorphism.
	\par Now, as $\mu(C_{\xi})<\infty$ for each $\xi<\o_{1}$, by Lemma \ref{Le1}, it follows that $L^{p}(C_{\xi})$ is separable. Moreover, by Proposition \ref{prop:purely-nonatomic-finite}, $(C_{\xi},\mu)$ is purely nonatomic. Recall that, for $1\leq p <\infty$, if $(X,\Sigma, \nu)$ is a nonzero finite, purely nonatomic measure space such that $L^{p}(\nu)$ is separable, then $L^{p}(\nu)$ is isometrically isomorphic to $L^{p}[0,1]$ (see, for instance, \cite[Cor.~p.~128]{La}). Hence, $L^{p}(C_{\xi})$ is isometrically isomorphic to $L^{p}[0,1]$ whenever $\mu(C_\xi)>0$, and is zero otherwise. There are $\mathfrak c$ positive-measure summands. Otherwise, CH would make their number countable and $L^p(\mu)$ separable, whereas the indicators of the cubes $\prod_n[a_n,a_n+1)$, $a\in\Z^\N$, form an uncountable family at mutual distance $2^{1/p}$. Therefore, we conclude that $L^{p}(\mu)$ is isometrically isomorphic to $\ell^{p}(\co, L^{p}[0,1])$.
\end{proof}

Note that since $L^{2}(\mu)$ is a Hilbert space, its isometric classification reduces to computing the Hilbert dimension, i.e., the cardinality of an orthonormal basis. By Theorem \ref{TeoPri}, this dimension is $\mf{c}$. Hence, by the standard classification of Hilbert spaces, $L^{2}(\mu)\cong \ell^{2}(\co)$.

Theorem \ref{TeoPri} provides a finite-dimensional reduction formula for computing the integral of any $\mu$-integrable function. For its statement we recall some preliminaries. Given a sequence of finite measure spaces $\{(X_{n},\Sigma_{n},\nu_{n})\}_{n\in\N}$ with $0<\nu_n(X_n)<\infty$ for every $n$ and $0<\prod_{n\in\N}\nu_n(X_n)<\infty$, consider the standard product measure $\nu:=\bigotimes_{n\in\N}\nu_{n}$. The construction is classical (see, for instance, \cite[\S 254]{Fremlin2} in the probability setting); the general case follows by normalizing each factor $\nu_{n}\mapsto \nu_{n}/\nu_{n}(X_{n})$. For this product measure, Jessen \cite{Je} (see also \cite[Th.~7.16]{St}) established a finite-dimensional reduction formula: for every $f\in L^{1}(\nu)$ there exists a sequence $f_{n}\in L^{1}(\nu^{n})$, where $\nu^{n}:=\bigotimes_{i=1}^{n}\nu_{i}$, such that
$$
\int_{\bigtimes_{i\in\N}X_{i}}f \; d\nu
=\lim_{n\to\infty}\int_{\bigtimes_{i=1}^{n}X_{i}}f_{n} \; d\nu^{n}.
$$
The final result of this section extends this formula to the measure $\mu$. For each $\xi<\omega_1$, let $D_\xi:=R_\xi$ if $\mu(R_\xi)>0$, and $D_\xi:=[0,1)^\N$ otherwise. Writing $D_\xi=\bigtimes_{i\in\N}D_\xi(i)$, set
$$
C^{n}_{\xi}:=\bigtimes_{i=1}^{n}D_{\xi}(i),\qquad n\in\N.
$$
Since $0<\mu(D_\xi)<\infty$, the restriction of $\mu$ to $D_{\xi}$ is the standard product measure $\bigotimes_{i\in\N}(\l\!\restriction D_{\xi}(i))$. For brevity, write $\mu_{\xi}^{n}:=\bigotimes_{i=1}^{n}(\l\!\restriction D_{\xi}(i))$ for every $n\in\N$.

\begin{theorem}
	Let $f\in L^{1}(\mu)$. Then, there exists a family of integrable functions 
	$$\{f_{\xi,n}\in L^{1}(\mu^{n}_{\xi}): \xi < \o_{1}, n\in\N\},$$
	such that
	\begin{equation}
		\label{E111}
		\int_{\mathbb{R}^{\mathbb{N}}}f \; d\mu =\lim_{n\to\infty}\sum_{\xi < \omega_{1}} \int_{C^{n}_{\xi}}f_{\xi, n} \; d\mu_{\xi}^{n}.
	\end{equation}
\end{theorem}

\begin{proof}
	Let $f\in L^{1}(\mu)$. Then, by Proposition \ref{Le5.1}, 
	$$
	\int_{\R^{\N}}|f| \; d\mu = \sum_{\xi<\o_{1}}\int_{C_{\xi}}|f|\; d\mu.
	$$
	Let $g_\xi:=f\chi_{C_\xi}$ on $D_\xi$. Then $\int_{D_\xi}g_\xi\,d\mu=\int_{C_\xi}f\,d\mu$. If $\mu(R_\xi)=0$, both sides vanish. For each $n$, define $f_{\xi,n}$ by integrating $g_\xi$ over the coordinates $i>n$ with respect to the finite product measure $\bigotimes_{i>n}(\lambda\restriction D_\xi(i))$. Fubini's theorem gives
	$$
	\int_{C_\xi^n} f_{\xi,n}\,d\mu_\xi^n=\int_{C_\xi} f\,d\mu,
	\qquad
	\int_{C_\xi^n}|f_{\xi,n}|\,d\mu_\xi^n\leq\int_{C_\xi}|f|\,d\mu.
	$$
	By Proposition \ref{Le5.1}, these integrals are absolutely summable over $\xi$. Summing the first identity proves \eqref{E111}. In fact, the sum has the required value for every $n$.
\end{proof}

\section{Study of the general case}\label{Se3}

In this section, we study the isometric structure of $L^{p}(\mu)$ in the general case, that is, if we do not impose the validity of CH. We prove that $L^{p}(\mu)$ contains an isometric and complemented copy of $\ell^{p}(\mf{c},L^{p}[0,1])$. This analysis is performed via a cube decomposition of $\R^{\N}$ that splits $L^{p}(\mu)$ into two complemented subspaces $G^{p}$ and $B^{p}$. 
\par Let us denote $I:=\Z^{\N}$. For each $\mathfrak{a}=(a_{n})_{n\in\mathbb{N}}\in I$, we introduce the unit cube
\begin{equation*}
	\mc{C}_{\mathfrak{a}}:=\bigtimes_{n\in\mathbb{N}}[a_{n},a_{n}+1).
\end{equation*}
Note that $\mc{C}_{\mf{a}}\in \mc{F}(\mc{B},\l)$ and $\mu(\mc{C}_{\mf{a}})=\vol(\mc{C}_{\mf{a}})=1$ for every $\mf{a}\in I$. 
Moreover, it is apparent that $\mc{C}_{\mf{a}}\cap \mc{C}_{\mf{b}}=\varnothing$ for each $\mf{a}\neq \mf{b}$ and 
$$
\R^{\N}= \bigsqcup_{\mf{a}\in I}\mc{C}_{\mf{a}}.
$$
The following result studies the behavior of a function $f\in L^{p}(\mu)$ with respect to this decomposition. 	
\begin{theorem}\label{T4.1}
	Let $1\leq p <\infty$ and $f\in L^{p}(\mu)$. Then the set
	\begin{equation*}
		\mathcal{O}_{f}:=\left\{\mathfrak{a}\in I : \int_{\mc{C}_{\mathfrak{a}}}|f|^{p} \  d\mu\neq 0 \right\}
	\end{equation*}
	\noindent is countable and 
	\begin{equation}\label{I1}
		\sum_{\mathfrak{a}\in\mathcal{O}_{f}}\int_{\mc{C}_{\mathfrak{a}}}|f|^{p} \ d\mu 
		=\int_{\mc{F}_{f}}|f|^{p} \ d\mu \leq \int_{\mathbb{R}^{\mathbb{N}}} |f|^{p} \ d\mu,
	\end{equation}
	where
	$$
	\mc{F}_{f}:= \bigsqcup_{\mf{a}\in \mc{O}_{f}}\mc{C}_{\mf{a}}.
	$$
\end{theorem}

\begin{proof}
	 If $\mathcal{O}_{f}$ were uncountable, then there would exist $\varepsilon>0$ and infinite countable $F\subset  \mathcal{O}_{f}$ such that 
	$ \int_{\mc{C}_{\mathfrak{a}}}|f|^{p} \  d\mu\geq \varepsilon$ for all $\mathfrak{a}\in F$. Then,
	\[
	\int_{\mathbb{R}^{\mathbb{N}}} |f|^{p} \ d\mu
	\geq \sum_{\mathfrak{a}\in F}\int_{\mc{C}_{\mathfrak{a}}} |f|^{p} \ d\mu
	\geq \sum_{\mathfrak{a}\in F}\varepsilon =\infty,
	\]
	This contradicts $f\in L^{p}(\mu)$. The proof is concluded as countability of $\mathcal{O}_{f}$ clearly implies inequality \eqref{I1}.
\end{proof}
It would be desirable that inequality \eqref{I1} were in fact an equality. However, as we shall see in Subsection \ref{S3}, this is not always true
(as we have uncountably many sets $\mc{C}_{\mathfrak{a}}$ and the integrals are only countably additive with respect to integrated sets).
This fact motivates the definition of the following subsets of $L^{p}(\mu)$, $1\leq p <\infty$:
$$
G^{p}:=\left\{f\in L^{p}(\mu) \; : \; \int_{\mathbb{R}^{\mathbb{N}}} |f|^{p} \ d\mu  = 
\sum_{\mathfrak{a}\in\mathcal{O}_{f}}\int_{\mc{C}_{\mathfrak{a}}}|f|^{p} \ d\mu =\int_{\mc{F}_{f}}|f|^{p} \ d\mu \right\}, 
$$
and
$$
B^{p}:=\left\{f\in L^{p}(\mu) \; : \; \int_{\mc{C}_{\mathfrak{a}}}|f|^{p} \ d\mu = 0 \; \; \text{for each } \mathfrak{a}\in I\right\}.
$$
The next two subsections are devoted to the study of these sets. 

\subsection{Study of $G^{p}$} \label{S2}

In this subsection, we proceed to study the subsets $G^{p}$, $1\leq p <\infty$. We are able to characterize completely its structure. We start by proving that $G^{p}$ is a subspace of $L^{p}(\mu)$. We need a preliminary lemma.

\begin{lemma}
	\label{L1}
	For every $f_{1}, f_{2}\in L^{p}(\mu)$, we have that $\mc{O}_{f_{1}+f_{2}} \subset \mc{O}_{f_{1}}\cup \mc{O}_{f_{2}}$, $\mc{F}_{f_{1}+f_{2}} \subset \mc{F}_{f_{1}}\cup \mc{F}_{f_{2}}$ and
	\begin{equation}
		\label{e3.2}
	\int_{  \mc{F}_{f_{1}+f_{2}}  }|f_{1}+f_{2}|^{p} \ d\mu=
	\int_{  \mc{F}_{f_{1}}\cup\mc{F}_{f_{2}}  }|f_{1}+f_{2}|^{p} \ d\mu.
	\end{equation}
\end{lemma}

\begin{proof}
	Let 
	$\mf{a}\in \mc{O}_{f_{1}+f_{2}}$. Then, by the triangle inequality,
	$$
	0<\left(\int_{\mc{C}_{\mf{a}}}|f_{1}+f_{2}|^{p} \ d\mu \right)^{\frac{1}{p}} \leq \left(\int_{\mc{C}_{\mf{a}}}|f_{1}|^{p} \ d\mu\right)^{\frac{1}{p}} + \left(\int_{\mc{C}_{\mf{a}}}|f_{2}|^{p} \ d\mu\right)^{\frac{1}{p}}.
	$$
	Therefore, 
	$$
	 \int_{\mc{C}_{\mf{a}}}|f_{1}|^{p} \ d\mu\neq 0 \quad \hbox{or} \;\;  \int_{\mc{C}_{\mf{a}}}|f_{2}|^{p} \ d\mu\neq 0,
	$$
	and this implies that $\mf{a}\in \mc{O}_{f_{1}}\cup \mc{O}_{f_{2}}$, completing the proof of inclusions. 
	Equation \eqref{e3.2} follows from the countability of $\mc{O}_{f_{1}}\cup\mc{O}_{f_{2}}$ and the fact that 
	$ \int_{\mc{C}_{\mf{a}}}|f_{1}+f_{2}|^{p}\ d\mu = 0$ for every $\mf{a}\in (\mc{O}_{f_{1}}\cup \mc{O}_{f_{2}})\setminus\mc{O}_{f_{1}+f_{2}}$.
%
%
%
\end{proof}

The following result proves that $G^{p}$ is a subspace of $L^{p}(\mu)$.

\begin{lemma}	\label{L2}
	$G^{p}$ is a subspace of $L^{p}(\mu)$.
\end{lemma}

\begin{proof}
	Let $f_{1}, f_{2}\in G^{p}$ and notice that 
	\begin{eqnarray*}
	\left(\int_{  (\mc{F}_{f_{1}}\cup\mc{F}_{f_{2}})^c  }|f_{1}+f_{2}|^{p} \ d\mu\right)^{1/p}
	&\leq &
	\left(\int_{  (\mc{F}_{f_{1}}\cup\mc{F}_{f_{2}})^c  }|f_{1}|^{p} \ d\mu\right)^{1/p}
	+
	\left(\int_{  (\mc{F}_{f_{1}}\cup\mc{F}_{f_{2}})^c  }|f_{2}|^{p} \ d\mu\right)^{1/p}\\
	&\leq &
	\left(\int_{  \mc{F}_{f_{1}}^c  }|f_{1}|^{p} \ d\mu\right)^{1/p}
	+
	\left(\int_{  \mc{F}_{f_{2}}^c  }|f_{2}|^{p} \ d\mu\right)^{1/p}=0,
        \end{eqnarray*}
        where the equality follows from the fact that $f_{1}, f_{2}\in G^{p}$. 
        So,
        $$
        \int_{  (\mc{F}_{f_{1}}\cup\mc{F}_{f_{2}})^c  }|f_{1}+f_{2}|^{p} \ d\mu = 0.
        $$
        By this and Lemma \ref{L1}, it follows that
	\begin{align*}
	\int_{\R^\N}|f_{1}+f_{2}|^{p} \ d\mu & =
	\int_{  \mc{F}_{f_{1}}\cup\mc{F}_{f_{2}}  }|f_{1}+f_{2}|^{p} \ d\mu+
	\int_{  (\mc{F}_{f_{1}}\cup\mc{F}_{f_{2}})^c  }|f_{1}+f_{2}|^{p} \ d\mu \\
	& =\int_{  \mc{F}_{f_{1}+f_{2}}  }|f_{1}+f_{2}|^{p} \ d\mu.
	\end{align*}
This proves that $f_{1}+f_{2}\in G^{p}$. If $f\in G^{p}$ and $k\in\R$, it is obvious that $k f\in G^{p}$. Therefore, $G^{p}$ is a subspace of $L^{p}(\mu)$.
\end{proof}

\noindent The next result characterizes the spaces $G^{p}$, $1\leq p < \infty$ up to isometric isomorphism. Firstly, we need some preliminary definitions. For each $\mathfrak{a}=(a_{n})_{n\in\mathbb{N}}\in I$, we define the translation $T_{\mathfrak{a}}: \R^{\N}  \to \R^{\N}$ as
\begin{equation*}
	T_{\mf{a}}(x_{n})_{n\in\mathbb{N}}:=(x_{n}-a_{n})_{n\in\mathbb{N}}.
\end{equation*}
It is clear that $T_{\mathfrak{a}}(\mc{C}_{\mathfrak{a}})=\mc{H}$, where $\mc{H}:=[0,1)^{\N}$ is the Hilbert cube. Moreover, since $\mu$ is translation invariant, we have that
\begin{equation}\label{II1}
	\mu(A)=\mu(T_{\mathfrak{a}}(A))=\mu(T_{\mathfrak{a}}^{-1}(A)) \text{ for each } A\in \mathcal{B}_{\infty}.
\end{equation}
The main result of this subsection is the following.
\begin{theorem}\label{T4.2}
	The following spaces are isometrically isomorphic for $1\leq p <\infty$,
	\begin{equation*}
		G^{p}\cong \ell^{p}(\co, L^{p}[0,1]).
	\end{equation*}
\end{theorem}

\begin{proof}
	Fix $1\leq p <\infty$ and define the operator
	\begin{equation*}
		\Psi:  G^{p}  \longrightarrow  \ell^{p}(I, L^{p}(\mc{H})), \quad \Psi(f):=\left(f\circ T^{-1}_{\mathfrak{a}}\right)_{\mathfrak{a}\in I}.
	\end{equation*}
	Given $f\in G^{p}$, we have
	\begin{align*}
		\|f\|^{p}_{L^{p}}=\int_{\mathbb{R}^{\mathbb{N}}}|f|^{p} \ d\mu= \sum_{\mathfrak{a}\in\mathcal{O}_{f}}\int_{\mc{C}_{\mf{a}}}|f|^{p} \ d\mu.
	\end{align*}
	For each $\mathfrak{a}\in I$, by the change of variable formula and identity \eqref{II1}, we deduce
	\begin{equation*}
		\int_{\mc{C}_{\mf{a}}}|f|^{p} \ d\mu =\int_{T_{\mathfrak{a}}(\mc{C}_{\mf{a}})}|(f\circ T^{-1}_{\mathfrak{a}})(x)|^{p} \ d\mu(T^{-1}_{\mathfrak{a}}(x))=\int_{\mc{H}}|f\circ T^{-1}_{\mathfrak{a}}|^{p} \ d\mu.
	\end{equation*}
	Therefore, 
	\begin{equation*}
		\|f\|^{p}_{L^{p}}= \sum_{\mathfrak{a}\in\mathcal{O}_{f}}\int_{\mc{C}_{\mf{a}}}|f|^{p} \ d\mu=\sum_{\mathfrak{a}\in\mathcal{O}_{f}}\int_{\mc{H}}|f\circ T^{-1}_{\mathfrak{a}}|^{p} \ d\mu= \sum_{\mf{a}\in I}\|f\circ T^{-1}_{\mf{a}}\|_{L^{p}(\mc{H})}^p=\|\Psi(f)\|_{\ell^{p}}^{p}.
	\end{equation*}
	Consequently $\Psi$ is well defined and it is an isometry. Finally, we see that $\Psi$ is onto. Let $\bar{f}=(f_{\mathfrak{a}})_{\mathfrak{a}\in I}\in \ell^{p}(I, L^{p}(\mc{H}))$ and denote by $\text{supp}(\bar{f})$ the support of $\bar{f}$, that is,
	$$
	\text{supp}(\bar{f}):=\{\mf{a}\in I \; : \; \|f_{\mf{a}}\|_{L^{p}(\mc{H})}\neq 0\}.
	$$
	Clearly, by the definition of $\ell^{p}(I, L^{p}(\mc{H}))$, the set $\text{supp}(\bar{f})$ is countable. Define
	\begin{equation*}
		f:=\sum_{\mathfrak{a}\in\text{supp}(\bar{f})}(f_{\mathfrak{a}}\circ T_{\mathfrak{a}})\cdot\chi_{\mc{C}_{\mathfrak{a}}}.
	\end{equation*}
	By the monotone convergence theorem and the change of variable formula, it follows that
	\begin{align*}
		\int_{\mathbb{R}^{\mathbb{N}}}|f|^{p} \ d\mu &= \int_{\mathbb{R}^{\mathbb{N}}} \sum_{\mathfrak{a}\in \text{supp}(\bar{f})}|f_{\mathfrak{a}}\circ T_{\mathfrak{a}}|^{p}\cdot\chi_{\mc{C}_{\mf{a}}} \ d\mu=\sum_{\mathfrak{a}\in\text{supp}(\bar{f})}\int_{\mc{C}_{\mf{a}}}|f_{\mathfrak{a}}\circ T_{\mathfrak{a}}|^{p} \ d\mu \\
		&=\sum_{\mathfrak{a}\in\text{supp}(\bar{f})}\int_{\mc{H}}|f_{\mathfrak{a}}|^{p} \ d\mu=\|\bar{f}\|_{\ell^{p}}^{p}<\infty.
	\end{align*}
	Therefore, $f\in L^{p}(\mu)$. Moreover, $f\in G^{p}$ since
		\begin{align*}
				\int_{\mathbb{R}^{\mathbb{N}}}|f|^{p} \ d\mu = \sum_{\mathfrak{a}\in\text{supp}(\bar{f})}\int_{\mc{C}_{\mf{a}}}|f_{\mathfrak{a}}\circ T_{\mathfrak{a}}|^{p} \ d\mu = \sum_{\mathfrak{a}\in\text{supp}(\bar{f})}\int_{\mc{C}_{\mf{a}}}|f|^{p} \ d\mu = \sum_{\mathfrak{a}\in\mc{O}_{f}}\int_{\mc{C}_{\mf{a}}}|f|^{p} \ d\mu,
		\end{align*}
	where in the last equality we have used that $\supp(\bar{f})=\mc{O}_{f}$. Indeed, the latter follows from the definition of $\supp(\bar{f})$ and $\mc{O}_{f}$ and the following identities
   \begin{equation*}
   	\int_{\mc{C}_{\mf{a}}}|f|^{p}\; d\mu = \int_{\mc{C}_{\mf{a}}}|f_{\mathfrak{a}}\circ T_{\mathfrak{a}}|^{p} \ d\mu = \int_{\mc{H}}|f_{\mf{a}}|^{p}\ d\mu = \|f_{\mf{a}}\|_{L^{p}(\mc{H})}^{p}.
   \end{equation*}
 It is straightforward to verify that $\Psi(f)=\bar{f}$. Hence, $\Psi$ is an isometric isomorphism. 
	\par Finally, as $\mu(\mc{H})=1$, by Lemma \ref{Le1}, it follows that $L^{p}(\mc{H})$ is separable. Therefore, by \cite[Cor.~p.~128]{La}, $L^{p}(\mc{H})$ is isometrically isomorphic to $L^{p}[0,1]$ for each $1\leq p <\infty$ and we conclude that $G^{p}$ is isometrically isomorphic to $\ell^{p}(\co, L^{p}[0,1])$. 
\end{proof}

\subsection{Study of $B^{p}$} \label{S3}

This subsection is devoted to the study of the sets $B^{p}$ for $1\leq p<\infty$. These sets contain the \textit{pathological} behavior of functions in $L^{p}(\mu)$. We describe some structure result of these sets. We start by proving that $B^{p}$ are Banach spaces.

\begin{lemma}
	For every $1\leq p <\infty$, $B^{p}$ is a Banach space.
\end{lemma}

\begin{proof}
	It is easy to verify that $B^{p}$ is a vector subspace of $L^{p}(\mu)$. It remains to prove that it is closed. Let $(f_{n})_{n\in\N}\subset B^{p}$ such that $f_{n}\to f$ as $n\to\infty$ in $L^{p}(\mu)$ for some $f\in L^{p}(\mu)$. Then, we obtain
	\begin{align*}
		\left|\int_{\mc{C}_{\mf{a}}}|f_{n}|^{p}\; d\mu - \int_{\mc{C}_{\mf{a}}}|f|^{p}\; d\mu\right|\leq \int_{\mc{C}_{\mf{a}}}|f_{n}-f|^{p} \; d\mu \leq \int_{\R^{\N}}|f_{n}-f|^{p}\; d\mu \xrightarrow[n\to\infty]{} 0.
	\end{align*}
	Therefore, we deduce
	$$
	0 = \lim_{n\to\infty}\int_{\mc{C}_{\mf{a}}}|f_{n}|^{p}\ d\mu = \int_{\mc{C}_{\mf{a}}} |f|^{p} \ d\mu,
	$$
	for each $\mf{a}\in I$. This proves that $f\in B^{p}$ and the proof is concluded.
\end{proof}

The main result of this section is the following.
\begin{theorem}
	For each $1\leq p < \infty$, $B^{p}$ contains an isometric and complemented copy of $\ell^{p}(\co,L^{p}[0,1])$.
\end{theorem}

\begin{proof}
	Recall that two infinite subsets $A, B\subset \N$ are called \textit{almost disjoint} if $|A\cap B|<\infty$. It is well known that there exists a family of almost disjoint subsets $\mc{A}$ of $\N$ with $|\mc{A}|=\mf{c}$.\footnote{For every $x\in \mathbb R$, fix a sequence of pairwise distinct rational numbers $(q_{x,n})_{n\in \N}$ that converges to $x$.
	Then, by considering a fixed bijection $f:\Q \to \N$, we have that the sets $A_x:=\{ f(q_{x,n}) : n\in \N \}$ form an almost disjoint family of subsets of $\N$ having cardinality $\mathfrak{c}$.}  For every $A\in \mc{A}$, let us denote 
	$$
	A:=\{n^{A}_{1}<n^{A}_{2}<\cdots\}.
	$$
	Then, for each $A\in\mc{A}$ and each $k\in\N$, let us consider two subintervals $I^{A}_{k}\subset (n^{A}_{k}-1,n^{A}_{k})$, $J^{A}_{k}\subset (-n^{A}_{k},1-n^{A}_{k})$ of length $2^{-k-1}$. Consider the Borel set
	$$
	\mc{I}_{A}:= \bigsqcup_{k\in\N}(I^{A}_{k} \sqcup J^{A}_{k})\in\mc{B}.
	$$
	Then, by the $\s$-additivity of the Lebesgue measure $\l$, we obtain that
	$$
	\l(\mc{I}_{A})=\sum_{k\in\N}(\l(I^{A}_{k})+\l(J^{A}_{k}))=\sum_{k\in\N}2^{-k}=1.
	$$
	For each $A\in\mc{A}$, consider the associated rectangle 
	$$
	C_{A}:=\bigtimes_{n\in\N}\mc{I}_{A}.
	$$
	Then, $C_{A}\in \mc{F}(\mc{B},\l)$ and
	$$
	\mu(C_{A})=\prod_{n\in\N}\l(\mc{I}_{A})=1.
	$$
	Note that given two distinct sets $A,B\in \mc{A}$, necessarily $|A\cap B|<\infty$ and, therefore, $\mc{I}_{A}\cap\mc{I}_{B}$ is a finite union of intervals. Consequently $\l(\mc{I}_{A}\cap\mc{I}_{B})<1$. This implies that
	\begin{equation}
		\label{e}
		\mu(C_{A}\cap C_{B})  = \mu\left(\left(\bigtimes_{n\in\N}\mc{I}_{A}\right)\cap \left(\bigtimes_{n\in\N}\mc{I}_{B}\right)\right)  = \mu \left( \bigtimes_{n\in\N}\mc{I}_{A}\cap\mc{I}_{B} \right) = \prod_{n\in\N}\l(\mc{I}_{A}\cap\mc{I}_{B})=0.
	\end{equation}
    Moreover, for each $\mf{a}\in I$ and each $A\in \mc{A}$, we have $\mu(\mc{C}_{\mf{a}}\cap C_{A})=0$. Indeed, by the definition of $\mc{I}_{A}$, for each $n\in \N$, there exists $k_{n}\in \N$ such that 
    	$$
    	\l([a_{n},a_{n}+1)\cap \mc{I}_{A})\leq 2^{-k_{n}},
    	$$
    where $\mc{C}_{\mf{a}}=\bigtimes_{n\in\N}[a_{n},a_{n}+1)$. Therefore,
    \begin{align*}
    	\mu(\mc{C}_{\mf{a}}\cap C_{A}) = \mu\left(\bigtimes_{n\in\N}[a_{n},a_{n}+1)\cap \mc{I}_{A}\right) = \prod_{n\in\N}\l([a_{n},a_{n}+1)\cap \mc{I}_{A})\leq \prod_{n\in\N}2^{-k_{n}}=0.
    \end{align*}
	On the one hand, note that, by definition, every $$\bar{f}=(f_{A})_{A\in\mc{A}}\in \left(\bigoplus_{A\in\mc{A}}L^{p}(C_{A})\right)_{p}$$ has countable support  $\text{supp}(\bar{f})$, where
	$$
	\text{supp}(\bar{f}):=\left\{A\in\mc{A} \; : \; \|f_{A}\|_{L^p(C_{A})} \neq 0\right\}.
	$$
	We introduce the linear operator
	$$
	T:\left(\bigoplus_{A\in\mc{A}}L^{p}(C_{A})\right)_{p} \longrightarrow B^{p}, \quad T(f_{A})_{A\in\mc{A}} := \sum_{A\in\text{supp}(\bar{f})}f_{A}\chi_{C_{A}}.
	$$
    Since for each $\mf{a}\in I$ and each $A\in\mc{A}$, the identity $\mu(\mc{C}_{\mf{a}}\cap C_{A})=0$ holds, we deduce that
    \begin{align*}
    \int_{\mc{C}_{\mf{a}}}|T(f_{A})_{A\in \mc{A}}|^{p} \; d\mu = \sum_{A\in\text{supp}(\bar{f})} \int_{\mc{C}_{\mf{a}}\cap C_{A}}|f_{A}|^{p} \; d\mu =0,
    \end{align*}
    which implies that $T(f_{A})_{A\in\mc{A}}\in B^{p}$. Moreover, 
	\begin{align*}
	\|T(f_{A})_{A\in \mc{A}}\|_{L^{p}}^{p}=\int_{\R^{\N}}\left| \sum_{A\in\text{supp}(\bar{f})}f_{A}\;\chi_{C_{A}}\right|^{p}\; d\mu = \sum_{A\in\text{supp}(\bar{f})} \|f_{A}\|^{p}_{L^{p}(C_{A})}=\|(f_{A})_{A}\|_{\ell^{p}}^{p}.
	\end{align*}
	This proves that $T$ is well defined and that it is an isometry. On the other hand, as $\mu(C_{A})=1$ for each $A\in\mc{A}$, by Lemma \ref{Le1}, $L^{p}(C_{A})$ is isometrically isomorphic to $L^{p}[0,1]$. Therefore,
	$$
	\left(\bigoplus_{A\in\mc{A}}L^{p}(C_{A})\right)_{p}\cong \ell^{p}(\mathfrak{c},L^{p}[0,1]).
	$$
	Consequently, $\ell^{p}(\mathfrak{c},L^{p}[0,1])$ embeds isometrically into $B^{p}$. 
	\par 	Finally, let us prove that 
	\begin{align*}
		V:=\text{Im}(T) =\left\{\sum_{A\in \supp(\bar{f})}f_{A}\chi_{C_{A}} \; : \; \bar{f}=(f_{A})_{A\in\mc{A}}\in \left(\bigoplus_{A\in\mc{A}}L^{p}(C_{A})\right)_{p}\right\} \cong\ell^{p}(\mathfrak{c},L^{p}[0,1])
	\end{align*}
	is a complemented subspace of $B^{p}$.
	Let us start by showing that for each $f\in B^{p}$, the set
	$$\mc{A}_{f}:=\left\{A\in \mc{A} \; : \; \|f \restriction C_{A}\|_{L^{p}(C_{A})}\neq 0\right\}$$ 
	is countable.
	If $\mc{A}_{f}$ were uncountable, then there would be $\varepsilon>0$ and infinite countable $F\subset  \mc{A}_{f}$ such that 
	$ \int_{C_{A}}|f|^{p} \  d\mu\geq \varepsilon$ for all $A\in F$. Then
	\[
	\int_{\mathbb{R}^{\mathbb{N}}} |f|^{p} \ d\mu
	\geq \sum_{A\in F}\int_{C_{A}} |f|^{p} \ d\mu
	\geq \sum_{A\in F}\varepsilon =\infty,
	\]
	contradicting $f\in L^{p}(\mu)$. Given $f\in B^{p}$, we define the sequence 
	$$
	\bar{f}:=(f_{A})_{A\in \mc{A}}, \quad f_{A}:=\left\{\begin{array}{ll}
		f \restriction C_{A} & \text{if} \; A\in \mc{A}_{f} \\
		0 & \text{if} \; A \notin \mc{A}_{f},
	\end{array}
\right.
	$$
	and the subset
	$$
	\mf{A}_{f}:=\bigcup_{A\in \mc{A}_{f}}C_{A}.
	$$
	Then $\bar{f}\in \left(\bigoplus_{A\in\mc{A}}L^{p}(C_{A})\right)_{p}$ since
	\begin{align*}
	\sum_{A\in\mc{A}}\|f_{A}\|^{p}_{L^{p}(C_{A})} & =\sum_{A\in\mc{A}_{f}}\|f_{A}\|^{p}_{L^{p}(C_{A})}= \sum_{A\in\mc{A}_{f}} \int_{C_{A}}|f|^{p} \; d\mu \\
		&= \int_{\R^{\N}} |f\chi_{\mf{A}_{f}}|^{p} \; d\mu \leq \int_{\R^{\N}} |f|^{p} \; d\mu = \|f\|_{L^{p}}^{p}<\infty,
	\end{align*}
    where we have used the countability of $\mc{A}_{f}$ and \eqref{e}. Moreover, $\supp(\bar{f})=\mc{A}_{f}$. Consequently, the linear operator $P: B^{p} \to V$ given by
	$$
	Pf:=f\chi_{\mf A_f}= \sum_{A\in \mc{A}_{f}}f _{A} \chi_{C_{A}},
	$$
	is well defined. For each $f\in B^{p}$, we have
	$$
	\|Pf\|_{L^{p}}^{p}=\int_{\R^{\N}} |f\chi_{\mf{A}_{f}}|^{p} \; d\mu \leq \int_{\R^{\N}} |f|^{p} \; d\mu = \|f\|_{L^{p}}^{p}.
	$$
	Then, $P$ is bounded. Finally, we prove that it is a projection. Let $f\in B^{p}$ and $B\in \mc{A}_{f}$. Then, denoting $g_{B}:=f_{B}\chi_{C_{B}}$, we have
	\begin{align*}
		\mc{A}_{g_{B}} & =\{A\in \mc{A} : \|f_{B}\chi_{C_{B}}\restriction C_{A}\|_{L^{p}(C_{A})}\neq 0\} \\
		&  = \{A\in \mc{A} : \|f\restriction (C_{A}\cap C_{B})\|_{L^{p}(C_{A})}\neq 0\} = \{B\},
	\end{align*}
	where the last equality follows from \eqref{e}. Therefore, 
	$$
	P(g_{B})=\sum_{A\in \mc{A}_{g_{B}}}(g_{B}\restriction C_{A})\chi_{C_{A}}=(g_{B}\restriction C_{B})\chi_{C_{B}}=g_{B},
	$$
	and this implies that
	\begin{align*}
		P^{2}f=P\left(\sum_{A\in\mc{A}_{f}}g_{A}\right)=\sum_{A\in\mc{A}_{f}}P(g_{A})=\sum_{A\in\mc{A}_{f}}g_{A}=Pf.
	\end{align*}
	Hence, $P$ is a projection and this proves that $V$ is complemented.
\end{proof}

The last result of this section gives the desired decomposition of $L^{p}(\mu)$ in the complemented subspaces $G^{p}$ and $B^{p}$.

\begin{theorem}
	For each $1\leq p <\infty$, the following decomposition into complemented subspaces holds,
	$$
	L^{p}(\mu)=G^{p}\oplus B^{p}\cong \ell^{p}(\co,L^{p}[0,1])\oplus B^{p}.
	$$
\end{theorem}
\begin{proof}
	Since $G^{p}$ and $B^{p}$ are closed subspaces, $G^p\cap B^p=\{0\}$, and $L^{p}(\mu)$ is Banach, it is enough to prove the algebraic decomposition. Each $f\in L^{p}(\mu)$ can be decomposed as
	$$
	f=f \chi_{\mc{F}_{f}} + f (1-\chi_{\mc{F}_{f}}).
	$$
	Let us prove that $f \chi_{\mc{F}_{f}}\in G^{p}$ and $f (1-\chi_{\mc{F}_{f}})\in B^{p}$. On the one hand, since $\mc{F}_{f}= \bigsqcup_{\mf{a}\in \mc{O}_{f}}\mc{C}_{\mf{a}}$ and $\mc{O}_{f}$ is countable, we have
	\begin{align*}
		\int_{\R^{\N}}|f\chi_{\mc{F}_{f}}|^{p} \; d\mu  = \int_{\mc{F}_{f}}|f|^{p} \; d\mu = \sum_{\mf{a}\in\mc{O}_{f}}\int_{\mc{C}_{\mf{a}}}|f|^{p}\; d\mu = \sum_{\mf{a}\in\mc{O}_{f}}\int_{\mc{C}_{\mf{a}}}|f\chi_{\mc{F}_{f}}|^{p}\; d\mu.
	\end{align*}
This proves that $f \chi_{\mc{F}_{f}}\in G^{p}$. On the other hand, if $\mf{a}\in\mc{O}_{f}$, then $\mc{C}_{\mf{a}}\cap \mc{F}_{f}^{c}=\varnothing$ and therefore
$$
\int_{\mc{C}_{\mf{a}}}|f(1-\chi_{\mc{F}_{f}})|^{p} \; d\mu = \int_{\mc{C}_{\mf{a}}}|f\chi_{\mc{F}_{f}^{c}}|^{p} \; d\mu  = 0.
$$
If $\mf{a}\notin \mc{O}_{f}$, by the definition of $\mc{O}_{f}$, $\int_{\mc{C}_{\mf{a}}}|f|^{p} \; d\mu=0$ and, consequently, 
$$
\int_{\mc{C}_{\mf{a}}}|f(1-\chi_{\mc{F}_{f}})|^{p} \; d\mu = 0.
$$
Hence, $\int_{\mc{C}_{\mf{a}}}|f(1-\chi_{\mc{F}_{f}})|^{p} \; d\mu = 0$ for every $\mf{a}\in I$, proving that $f(1-\chi_{\mc{F}_{f}})\in B^{p}$. This concludes the proof.
\end{proof}

Therefore, we obtain the following.

\begin{corollary}
	For $1\leq p <\infty$, the space $L^{p}(\mu)$ contains an isometric and complemented copy of $\ell^{p}(\co,L^{p}[0,1])$.
\end{corollary}

\section{Characterization of isometries and statement of the conjecture}\label{Se4}

In this final section, we determine when $L^{p}(\nu)$, for a relatively nonatomic measure $\nu$ (see Appendix \ref{AP}), is isometrically isomorphic to $\ell^{p}(\kappa, L^{p}[0,1])$ for some cardinal $\kappa$; symbolically,
\begin{equation*}
	L^{p}(\nu)\cong\ell^{p}(\kappa, L^{p}[0,1]).
\end{equation*}
In addition, in the case where such an isomorphism exists, we describe all isometric isomorphisms between these spaces. We conclude by applying these results to conjecture that the isometric isomorphism $L^{p}(\mu)\cong \ell^{p}(\co,L^{p}[0,1])$ is independent of ZFC.
\par We begin by introducing the notion of a $\kappa$-separable partition.

\begin{definition}[$\kappa$-\textbf{separable partition}]
	\label{De4.2}
	Let $(X,\Sigma,\nu)$ be a measure space and let $\kappa$ be a cardinal. A family $\{X_i\colon i\in J\}\subset \Sigma$ of $\s$-finite separable subsets with nonzero measure is a $\kappa$-\textit{separable partition} if $|J|=\kappa$,
	\begin{equation}
		\label{P11}
		\nu(X_i \cap X_j)=0 \; \;  \text{for distinct} \; \; i,j  \in J,
	\end{equation}
	and for every $E\in\Sigma$ with finite measure,
	\begin{equation}
		\label{P}
		\nu(E)=\sum_{i\in J} \nu(E\cap X_i).
	\end{equation}
\end{definition}

Under CH, it is straightforward to verify that the family $\{C_\xi\colon \xi<\o_{1},\ \mu(C_\xi)>0\}$ introduced in Proposition \ref{Pr2.1} is a $\mf{c}$-separable partition of $(\R^{\N},\mu)$.

The first result of this section shows that if a measure space admits a $\kappa$-separable partition, then $L^{p}(\nu)\cong \ell^{p}(\kappa,L^{p}[0,1])$. We begin with the following lemma.

\begin{lemma}
	\label{L}
	Let $(X,\Sigma,\nu)$ be nonzero, purely nonatomic, $\s$-finite, and separable. Then, for every $1\leq p <\infty$, $L^{p}(\nu)$ is isometrically isomorphic to $L^{p}[0,1]$.
\end{lemma}

\begin{proof}
	By hypothesis, discarding a null set and subdividing positive-measure pieces if necessary, there is a partition $\{X_{n}: n\in \N\}\subset \Sigma$ of $X$ into purely nonatomic, separable sets with $0<\nu(X_{n})<\infty$ for each $n\in\N$. By \cite[Cor.~p.~128]{La}, for each $n\in\N$ there is an isometric isomorphism $R_{n}: L^{p}(X_{n})\to L^{p}[0,1]$. Therefore,
	$$
	R: \left(\bigoplus_{n\in\N}L^{p}(X_{n})\right)_{p}\longrightarrow \ell^{p}(\N,L^{p}[0,1]), \quad R[(f_{n})_{n\in\N}]:=(R_{n}(f_{n}))_{n\in\N},
	$$
	is an isometric isomorphism. Finally, the standard isometries
	$$
	T_{1}:\left(\bigoplus_{n\in\N}L^{p}(X_{n})\right)_{p}\longrightarrow L^{p}(\nu), \quad T_{1}[(f_{n})_{n\in\N}]:=\sum_{n\in\N}f_{n}\chi_{X_{n}}, 
	$$
	and $	T_{2}:\ell^{p}(\N,L^{p}[0,1])\to L^{p}[0,1]$ given by
	$$
	T_{2}[(f_{n})_{n\in\N}](x):=\left\{
	\begin{array}{ll}
		2^{\frac{n}{p}}f_{n}(2^{n}x-1), & \text{if} \;\; x\in (2^{-n},2^{-(n-1)}), \\
		0, & \text{otherwise},
	\end{array}
	\right.
	$$
	complete the proof since $T_{2}\circ R\circ T_{1}^{-1}: L^{p}(\nu)\to L^{p}[0,1]$ is the required isometric isomophism.
\end{proof}

\begin{theorem}
	\label{Th4.3}
	Let $1\leq p < \infty$. If a relatively nonatomic measure space $(X,\Sigma,\nu)$ admits a $\kappa$-separable partition, then $L^{p}(\nu)$ is isometrically isomorphic to $\ell^{p}(\kappa,L^{p}[0,1])$.
\end{theorem}

\begin{proof}
	Let $\{X_i\colon i\in J\}$ be a $\kappa$-separable partition and consider the operator
	\begin{equation}
		\label{I2}
		T: L^{p}(\nu) \longrightarrow \left(\bigoplus_{i\in J} L^{p}(X_i)\right)_{p}, \quad Tf := (f \restriction X_i)_{i\in J}.
	\end{equation}
	The operator \eqref{I2} is well defined.
	Indeed, for each finite $F\subset J$ and $i\in F$, by \eqref{P11} we have 
	$$
	\nu \left( X_i \cap \bigcup_{j\in F\setminus\{i\}} X_j \ \right) = \nu \left( \bigcup_{j\in F\setminus \{i\}} (X_i \cap X_j) \ \right) \leq \sum_{j\in F\setminus\{i\}} \nu (X_{i}\cap X_{j})= 0.
	$$
	Hence, for any finite $F\subset J$ and $f\in L^{p}(\nu)$,
	$$
	\sum_{i\in F}\int_{X_i}|f|^{p} \; d\nu \leq \int_{X} |f|^{p} \; d\nu.
	$$
	It follows that, for every $f\in L^{p}(\nu)$,
	\begin{align*}
		\|Tf\|^{p}_{\ell^{p}} & =\sum_{i\in J}\|f\restriction X_i\|^{p}_{L^{p}(X_i)} = \sum_{i\in J}\int_{X_i}|f|^{p} \; d\nu \\
		& = \sup \left\{ \sum_{i\in F}\int_{X_i}|f|^{p} \; d\nu :  F\subset J \text{ finite} \right\} \leq \int_{X} |f|^{p} \; d\nu <\infty.
	\end{align*}
	For every $E\in\Sigma$ with finite measure, consider the characteristic function $\chi_{E}$. Then, by \eqref{P}, we obtain
	\begin{align*}
		\|\chi_{E}\|^{p}_{L^{p}} & =\int_{X}|\chi_{E}|^{p} \; d\nu = \nu(E) = \sum_{i\in J}\nu(E\cap X_i)=\sum_{i\in J}\int_{X_i}|\chi_{E}|^{p} \; d\nu \\
		&  = \sum_{i\in J} \|\chi_{E} \restriction X_i\|^{p}_{L^{p}(X_i)} = \|T\chi_{E}\|^{p}_{\ell^{p}}.
	\end{align*}
	This extends to functions in $\mc{X}:=\spann\{\chi_{E} : E\in \Sigma, \; \nu(E)<\infty\}$: for each $f\in\mc{X}$, 
	\begin{equation}
		\label{P1}
		\|f\|_{L^{p}}=\|Tf\|_{\ell^{p}}.
	\end{equation}  
	Now take a sequence $(f_{n})_{n\in \N}\subset \mc{X}$ satisfying \eqref{P1} with $f_{n}\to f$ in $L^{p}$. Then, 
	\begin{align*}
		\|f\|^{p}_{L^{p}}=\lim_{n\to\infty}\int_{X}|f_{n}|^{p} \; d\nu = \lim_{n\to\infty} \sum_{i\in J}\int_{X_i}|f_{n}|^{p} \; d\nu.
	\end{align*}
	Since each sum is finite, there exist countable sets $J_{n}\subset J$ such that $\|f_{n}\|_{L^{p}(X_i)}=0$ for all $i\in J\setminus J_n$. 
	Let $\hat{J}:=\bigcup_{n\in\N}J_{n}$ (countable) and set 
	$$\mc{J}:=\bigcup_{i \in \hat{J}}X_i \in \Sigma.$$
	By convergence, $\|f\|_{L^{p}(X_i)}=0$ for each $i\in J\setminus \hat{J}$, hence
	$$
	\int_{\mc{J}} |f|^p \; d\nu = \sum_{i\in \hat{J}} \int_{X_i} |f|^p \; d\nu = \sum_{i\in J} \int_{X_i} |f|^p \; d\nu.
	$$
	Therefore,
	\begin{align*}
		\|f\|^{p}_{L^{p}} = \lim_{n\to\infty} \sum_{i \in J_{n}}\int_{X_i}|f_{n}|^{p} \; d\nu = \lim_{n\to\infty} \int_{\mc{J}} |f_{n}|^{p} \; d\nu = \int_{\mc{J}}|f|^{p} \; d\nu = \sum_{i \in J} \int_{X_i}|f|^{p} \; d\nu = \|Tf\|^{p}_{\ell^{p}}.
	\end{align*}
	By density of $\mc{X}$ in $L^{p}(\nu)$ we conclude that \eqref{P1} holds for every $f\in L^{p}(\nu)$, so $T$ is an isometry. To see that $T$ is onto, let 
	$$(f_i)_{i \in J}\in \left(\bigoplus_{i \in J} L^{p}(X_i)\right)_{p}.$$
	Then there exists a countable $C\subset J$ with $\|f_i\|_{L^{p}(X_i)}=0$ for all $i \in J\setminus C$ 
	(and in particular $f_i \restriction X_i = 0$ $\nu$-a.e. for $i\in J\setminus C$).
	For each $i\in C$,
	\begin{equation*}
		\nu \left( X_i \cap \bigcup_{j\in C\setminus \{ i \}} X_j \right) = \nu \left( \bigcup_{j\in C\setminus \{ i \}} (X_i \cap X_j) \right) \leq \sum_{j\in C\setminus \{ i \}} \nu( X_i \cap X_j ) = 0.
	\end{equation*}
	Define $f: X \to \R$ by
	$$
	f:= \sum_{i \in C}f_i \;  \chi_{X_i}.
	$$
	Then $f\restriction X_i = f_i \restriction X_i$ $\nu$-a.e. for any $i\in J$, and
	$f\in L^{p}(\nu)$ since
	$$
	\|f\|^{p}_{L^{p}}=\int_{X}|f|^{p} \; d\nu = \sum_{i \in C}\int_{X_i}|f_i|^{p}\; d\nu = \|(f_i)_{i \in J}\|^{p}_{\ell^{p}}<\infty.
	$$
	Hence $Tf=(f_i)_{i \in J}$, so $T$ is an isometric isomorphism. Finally, observe that for each \(i\in J\), the restricted measure space \((X_i,\nu)\) is purely nonatomic because it is \(\sigma\)-finite; this follows directly from the assumption that \((X,\Sigma,\nu)\) is relatively nonatomic. Therefore, $X_{i}$ is separable, $\s$-finite and purely nonatomic. By Lemma \ref{L}, we have $L^{p}(X_{i})\cong L^{p}[0,1]$, and therefore
	$$
	\left(\bigoplus_{i\in J} L^{p}(X_i)\right)_{p}\cong \ell^{p}(\kappa, L^{p}[0,1]),
	$$
	as $|J|=\kappa$. This completes the proof.
\end{proof}

The nonatomic assumption on $(X,\Sigma,\nu)$ is indeed necessary. For example, on the atomic measure space $(\N, \mc{P}(\N), \nu)$ with $\nu$ the counting measure, the family $\{\{n\} : n\in \N\}$ is an $\aleph_{0}$-separable partition, whereas $L^{p}(\N,\nu) = \ell^{p}$ is not isometrically isomorphic to $\ell^{p}(\N, L^{p}[0,1])\cong L^{p}[0,1]$ for $p\neq 2$.

\vspace{2pt}

The second part of this section provides a converse to Theorem \ref{Th4.3}. We begin with the notion of a $\kappa$-separable envelope.

\begin{definition}[$\kappa$-\textbf{separable envelope}]
	\label{De4.1}
	Let $(X,\Sigma,\nu)$ be a measure space and let $\kappa$ be a cardinal. A family $\{X_i\colon i\in J\}\subset \Sigma$ of $\s$-finite separable subsets with nonzero measure is a $\kappa$-\textit{separable envelope} if $|J|=\kappa$,
	\begin{equation}
		\label{P12}
		\nu(X_i\cap X_j)=0 \; \;  \text{for distinct} \; \; i,j \in J,
	\end{equation}
	and for every $E\in\Sigma$ of finite measure with $E\subset \bigcup_{i \in J}X_i$,
	\begin{equation}
		\label{E}
		\nu(E)=\sum_{i\in J} \nu(E\cap X_i).
	\end{equation}
\end{definition}

The key difference from a \(\kappa\)-separable partition (Definition \ref{De4.2}) is that here the additivity identity \(\nu(E)=\sum_{i\in J}\nu(E\cap X_i)\) is required only for sets \(E\) contained in \(\bigcup_{i\in J}X_i\); the family need not exhaust, even up to null sets, all finite-measure subsets of \(X\). Consequently, every \(\kappa\)-separable partition is a \(\kappa\)-separable envelope.

We say that a measure space $(X,\Sigma,\nu)$ has the \textit{separability property} if every finite-measure set $B\in\Sigma$ is separable. Note that if $(X,\Sigma,\nu)$ has the separability property, then every $\s$-finite measurable subset $B\in \Sigma$ is also separable. If, in addition, $\nu$ is relatively nonatomic and $\nu(B)>0$, then $L^p(B)\cong L^p[0,1]$, as shown in Lemma \ref{L}. Of course, $(\R^{\N},\mu)$ has the separability property by Lemma \ref{Le1}.

The following result is essentially the converse of Theorem \ref{Th4.3} for measure spaces with the separability property. In this case we also fully describe all isometric isomorphisms between \(L^{p}(\nu)\) and \(\ell^{p}(\kappa, L^{p}[0,1])\).

\begin{theorem}
	\label{TP3}
	Let $(X,\Sigma,\nu)$ be a relatively nonatomic measure space with the separability property, and $1\leq p < \infty$, $p\neq 2$. If $L^{p}(\nu)$ is isometrically isomorphic to $\ell^{p}(\kappa,L^{p}[0,1])$, then $(X,\Sigma,\nu)$ admits a $\kappa$-separable envelope $\{X_{i} : i\in J\}$ with $|J|=\kappa$. Moreover, every isometric isomorphism $T:L^{p}(\nu)\to\ell^{p}(\kappa,L^{p}[0,1])$ is of the form
	\begin{equation}
		\label{Fac}
		T(f)=(R^{-1}_{i}(\theta(f)\restriction X_{i}))_{i\in J},
	\end{equation}
	for some isometry $\theta: L^{p}(\nu)\to L^{p}(\nu)$ and a family $(R_{i})_{i\in J}$ of isometric isomorphisms $R_{i}:L^{p}[0,1]\to L^{p}(X_{i})$.
\end{theorem}

\begin{proof}
	Suppose $L^{p}(\nu)\cong \ell^{p}(\kappa,L^{p}[0,1])$. Then there exists an isometric isomorphism
	$$
	T: L^{p}(\nu) \longrightarrow \ell^{p}(\kappa,L^{p}[0,1]), \quad Tf:=(f_{j})_{j\in J},
	$$
	for some set $J$ with $|J|=\kappa$. For each $i\in J$ define $T_{i}:=\pi_{i}\circ T: L^{p}(\nu)\to L^{p}[0,1]$, where 
	$$
	\pi_{i}: \ell^{p}(\kappa,L^{p}[0,1])\longrightarrow L^{p}[0,1], \quad \pi_{i}[(f_{j})_{j\in J}]:=f_{i}.
	$$
	For $i \in J$, let $g_{i}:=T^{-1}(\mf{e}_{ij})_{j\in J}\in L^{p}(\nu)$, where $(\mf{e}_{ij})_{j\in J}\in \ell^{p}(\kappa,L^{p}[0,1])$ is given by
	$$
	\mf{e}_{ij}:=\left\{
	\begin{array}{ll}
		1 & \text{if } i=j, \\
		0 & \text{if } i\neq j.
	\end{array}
	\right.
	$$
	Set $X_{i}:=\supp(g_{i})\in \Sigma$ for $i\in J$. We show that $\nu(X_{i}\cap X_{j})=0$ for $i\neq j$. Since $T$ is an isometry,
	\begin{equation}
		\label{e1}
		\|g_{i}\|_{L^{p}}^{p}=\sum_{j\in J}\|T_{j}g_{i}\|^{p}_{L^{p}[0,1]}=\sum_{j\in J}\|\mf{e}_{ij}\|^{p}_{L^{p}[0,1]}=1,
	\end{equation}
	for any $i \in J$, and
	\begin{align*}
		\|g_{i}\pm g_{j}\|^{p}_{L^{p}}=\sum_{k\in J}\|T_{k}g_{i}\pm T_{k}g_{j}\|^{p}_{L^{p}[0,1]}=\sum_{k\in J}\|\mf{e}_{ik}\pm \mf{e}_{jk}\|^{p}_{L^{p}[0,1]}=2
	\end{align*}
	whenever $i\neq j$.
	Identity \eqref{e1} implies $\nu(X_{i})>0$ for each $i$ and that $X_{i}$ is $\s$-finite. Moreover, for $i\neq j$, $g_{i}, g_{j}\in L^{p}(X_{i}\cup X_{j})$ attain equality in Clarkson’s inequality,
	$$
	\|g_{i}+g_{j}\|^{p}_{L^{p}}+\|g_{i}-g_{j}\|^{p}_{L^{p}}=2\left(\|g_{i}\|^{p}_{L^{p}}+\|g_{j}\|^{p}_{L^{p}}\right).
	$$
	Since $X_{i}\cup X_{j}$ is $\s$-finite and $p\neq 2$, \cite[Cor. 2.1]{Lam} yields $g_{i}\cdot g_{j}=0$ $\nu$-a.e. on $X_{i}\cap X_{j}$, hence $\nu(X_{i}\cap X_{j})=0$. As $(X,\Sigma,\nu)$ is relatively nonatomic with the separability property, each $X_{i}$ is purely nonatomic, separable, and $\s$-finite. By Lemma \ref{L}, for each $i\in J$, there exists an isometric isomorphism $R_{i}:L^{p}[0,1]\to L^{p}(X_{i})$. Composing with $T$ we obtain
	$$
	\hat{T}:L^{p}(\nu) \longrightarrow \left(\bigoplus_{i\in J} L^{p}(X_{i})\right)_{p}, \quad \hat{T}(f):=\left((R_{i}\circ T_{i})(f)\right)_{i\in J},
	$$
	an isometric isomorphism. Consider
	$$
	R: \left(\bigoplus_{i\in J} L^{p}(X_{i})\right)_{p} \longrightarrow L^{p}(\nu)
	, \quad	R[(f_{i})_{i\in J}]:=\sum_{i\in J_{\mf{f}}}f_{i}\; \chi_{X_{i}},
	$$
	where $\mf{f}:=(f_{i})_{i\in J}$ and $J_{\mf{f}}:=\{i\in J : \|f_{i}\|_{L^{p}(X_{i})}\neq 0\}$ is countable. This map is well defined and isometric since
	$$
	\|R[(f_{i})_{i\in J}]\|_{L^{p}}^{p}=\int_{X}\sum_{i\in J_{\mf{f}}}|f_{i}|^{p}\chi_{X_{i}} \; d\nu = \sum_{i\in J_{\mf{f}}} \int_{X_{i}}|f_{i}|^{p} \; d\nu = \|(f_{i})_{i\in J}\|_{\ell^{p}}^{p}.
	$$
	Define $\theta:= R\circ \hat{T}:L^{p}(\nu)\to L^{p}(\nu)$.  For every $f\in L^{p}(\nu)$,
	\begin{equation}
		\label{E11}
		\theta(f)=(R\circ \hat{T})(f)=\sum_{i\in J_{\mf{f}}}f_{i} \; \chi_{X_{i}},
	\end{equation}
	where $\hat{T}(f)=(f_{i})_{i\in J}=\mf{f}$. Thus there exists $N\in\Sigma$ with $\nu(N)=0$ such that
	$$
	\theta(f)(x)=\sum_{i\in J_{\mf{f}}}f_{i} \; \chi_{X_{i}}(x), \quad \text{for every } x\in X\setminus N.
	$$
	Since 
	$$
	\nu\left(X_{i}\cap \left( \bigcup_{j\in J_{\mf{f}}\setminus \{i\}} X_{j}\cup N\right)\right)\leq\nu\left(\bigcup_{j\in J_{\mf{f}}\setminus \{i\}} (X_{i}\cap X_{j}) \right)+ \nu(X_{i}\cap N)=0,
	$$
	we have $f_{i}=\theta(f)\restriction X_{i}$ $\nu$-a.e., and hence
	\begin{equation}
		\label{JC}
		\hat{T}(f)=(\theta(f)\restriction X_{i})_{i\in J}.
	\end{equation}
	We now prove that $\{X_{i} : i\in J\}$ is a $\kappa$-separable envelope. Let $E\in \Sigma$ satisfy $\nu(E)<\infty$ and $E\subset \bigcup_{i\in J}X_{i}$. Then $(\chi_{E}\restriction X_{i})_{i\in J}\in \left(\bigoplus_{i\in J} L^{p}(X_{i})\right)_{p}$ since
	$$
	\|(\chi_{E}\restriction X_{i})_{i\in J}\|^{p}_{\ell^{p}}=\sum_{i\in J}\nu(E\cap X_{i}):=\sup\left\{\sum_{j\in F}\nu(E\cap X_{j}) : F\subset J \text{ finite}\right\}\leq \nu(E)<\infty.
	$$
	As $\hat{T}$ is onto, there exists $f_{E}\in L^{p}(\nu)$ such that $\hat{T}(f_{E})=(\chi_{E}\restriction X_{i})_{i\in J}$. Let
	$$
	J_{E}:=\{i\in J : \|\chi_{E}\restriction X_{i}\|_{L^{p}(X_{i})}\neq 0\}=\{i\in J : \nu(E\cap X_{i})\neq 0\}.
	$$
	Set $E_0:=E\setminus\bigcup_{i\in J_E}X_i$. Then $\nu(E_0\cap X_i)=0$ for every $i$. By the same disjointness criterion used above, $T$ preserves disjoint supports. Thus $T\chi_{E_0}$ is disjoint from $Tg_i$ for every $i$, so all its coordinates vanish. Since $T$ is injective, $\nu(E_0)=0$.
	Since $J_{E}$ is countable and $E_0$ is null, equation \eqref{E11} gives
	$$
	\theta(f_{E})=\sum_{i\in J_{E}}\chi_{E\cap X_{i}}=\chi_{E},
	$$ 
	and because $\theta$ is an isometry,
	$$
	\|f_{E}\|^{p}_{L^{p}}=\|\theta(f_{E})\|^{p}_{L^{p}}=\|\chi_{E}\|^{p}_{L^{p}}=\nu(E).
	$$
	Finally, as $\hat{T}$ is also an isometry,
	$$
	\nu(E)=\|f_{E}\|^{p}_{L^{p}}=\|\hat{T}(f_{E})\|^{p}_{\ell^{p}}=\|(\chi_{E}\restriction X_{i})_{i\in J}\|^{p}_{\ell^{p}}=\sum_{i\in J}\nu(E\cap X_{i}).
	$$
	Thus $\{X_{i} : i\in J \}$ is indeed a $\kappa$-separable envelope. The factorization \eqref{Fac} follows directly from \eqref{JC}. This completes the proof.
\end{proof}

The case $p=2$, excluded in Theorem \ref{TP3}, is simpler. Here $L^{2}(\nu)\cong \ell^{2}(\kappa)$, where $\kappa$ is the cardinality of a maximal orthonormal set in $L^{2}(\nu)$. If $\kappa>0$, relative nonatomicity implies that $\kappa$ is infinite, and
$$
L^{2}(\nu)\cong \ell^{2}(\kappa) \cong \ell^{2}(\kappa,\ell^{2})\cong \ell^{2}(\kappa, L^{2}[0,1]).
$$
Hence, if $\kappa$ is the cardinality of a maximal orthonormal set in $L^{2}(\nu)$, the isometric isomorphism $L^{2}(\nu)\cong\ell^{2}(\kappa, L^{2}[0,1])$ always holds (the case $\kappa=0$ is immediate).
\par We conclude the article with the following conjecture.
\begin{conjecture}
	\label{C4}
	There is a model of ZFC in which $(\R^{\N},\mu)$ does not admit a $\mf{c}$-separable envelope.
\end{conjecture}
Assuming this conjecture, since \((\mathbb{R}^{\mathbb{N}},\mu)\) is relatively nonatomic and has the separability property, Theorem~\ref{TP3} yields that the statement
\[
L^{p}(\mu)\cong \ell^{p}(\mathfrak{c},L^{p}[0,1]) \quad (p\neq 2)
\]
is independent of ZFC (i.e., neither provable nor refutable within ZFC).

\medskip
\noindent\textbf{Note added in revision.} In our subsequent paper \cite{RSnote}, we prove that $L^p(\mu)\cong\ell^p(\mathfrak c,L^p[0,1])$ in ZFC for every $1\leq p<\infty$. This resolves Conjecture \ref{C4} negatively. The conjecture and its discussion above are retained in their original historical form.

\vspace{10pt}
\noindent \textbf{Acknowledgments:} The authors would like to thank Prof. Krzysztof C. Ciesielski for his insightful and meaningful comments, which contributed to the main result of this paper, Theorem \ref{TeoPri}. His feedback also provided valuable input toward the proof of Theorem \ref{A1} in the Appendix section and helped improve the formulation of Theorem \ref{A2}. Additionally, we appreciate his suggestions, which enhanced the overall presentation of this paper.
\par We also express our gratitude to Prof. Mingu Jung for his helpful comments regarding the classification of isometric isomorphisms in Theorem \ref{TP3}, which contributed to clarifying our discussion on this topic.
\par We also thank two anonymous referees whose insightful comments helped improve considerably the presentation, scope and results of this paper.

\appendix

\section{Some properties of the measure $\mu$}\label{AP}

In this appendix, we prove some properties and curiosities of the measure $\mu$ not present in the literature.
\par We start by proving that the measure $\mu$ is not localizable. Recall that a measure space $(X,\Sigma,\nu)$ is localizable if $\nu$ is semi-finite and every $\mc{E}\subset \Sigma$ admits an essential supremum, that is, there is $H_{\mc{E}}\in\Sigma$ satisfying:
\begin{itemize}
	\item[{\rm (a)}] $\nu(E\setminus H_{\mathcal E})=0$ (i.e., $E\subset^* H_{\mathcal E}$) 
	for every $E\in\mathcal E$; 
	\item[{\rm (b)}] if $G\in \Sigma$ is such that  $\nu(E\setminus G)=0$ (i.e., $E\subset^* G$) 
	for every $E\in\mathcal E$, then 
	$\nu(H_{\mathcal E}\setminus G)=0$ (i.e., $H_{\mathcal E}\subset^*  G$).
\end{itemize}

\begin{theorem}
	\label{A1}
	The measure space $(\R^{\N},\mu)$ is not localizable.
\end{theorem}

\begin{proof}
	To see this, first notice that $\mc{B}_{\infty}$ contains the following family $\EEE$ of continuum many pairwise disjoint sets of positive finite $\mu$-measure 
	(in fact, of measure 1):
	\[
	\EEE=\left\{\prod_{i\in\mathbb{N}}[s_i,s_i+1]\colon s\in\{0,2\}^\mathbb{N}\right\}.
	\]
	Now, by way of contradiction, assume that $\mu$ is localizable
	and for every $\EE\in{\mathcal P}(\EEE)$, let
	$H_\EE$ be its essential supremum. 
	Since $|{\mathcal P}(\EEE)|=2^{2^{\omega}}>2^\omega=|\mc{B}_{\infty}|$, 
	there are distinct $\EE,\EE'\in {\mathcal P}(\EEE)$ with $H_\EE=H_{\EE'}$. 
	Since $\EE\neq \EE'$, there is an $E\in \EEE$ that belongs to one of them, but not the other, say 
	$E\in \EE\setminus \EE'$. 
	Then we must have $E\subset^* H_\EE$.
	At the same time, $H_\EE=H_{\EE'}\subset^* H_\EE\setminus E$,
	a contradiction, since $H_\EE\not\subset^* H_\EE\setminus E$. 
\end{proof}

The proof of Theorem \ref{A1} shows that there is $\mc{E}\subset \Sigma$ with no essential supremum. Moreover, we will prove in Proposition \ref{NF} that $\mu$ is neither semi-finite.
	
The second part of this appendix studies the atomic properties of $(\R^{\N},\mu)$. The next result shows that $(\R^{\N},\mu)$ is not purely nonatomic.

\begin{proposition}
	\label{HY}
	There exist measurable subsets \(A\subset\mathbb{R}^{\mathbb{N}}\) that are atoms for \(\mu\).
\end{proposition}

\begin{proof}
	Consider the hyperplane
	\begin{equation}
		\label{H}
	H := \{(x_{i})_{i\in\N}\in\mathbb{R}^{\mathbb{N}}: x_1=0\}.
	\end{equation}
	It is Borel since it is the preimage of \(\{0\}\) under the continuous projection \(x\mapsto x_1\).
	We claim that \(\mu(H)=\infty\) and that \(H\) is an atom.
	Suppose by contradiction that \(H\) can be covered by a countable family \((R_n)_{n\in\N}\subset\mathcal{F}(\mc{B},\l)\).
	Write 
	$$R_n=\bigtimes_{i\in \N}R^{n}_{i}, \quad R^{n}_{i}\in \mc{B}.$$
	Since $R_{n}\in\mc{F}(\mc{B},\l)$, necessarily $\l(R^{n}_{i})<\infty$ for every $i\in\N$. Choose, for every $n\in\N$, points
	$x_{n+1}\in\mathbb{R}\setminus R_{n+1}^{n}$. Note that this is possible since \(\lambda(R_{n+1}^{n})<\infty\) while \(\lambda(\mathbb{R})=\infty\).
	Define $(y_{i})_{i\in\N}\in\R^{\N}$ by
	$$
	y_{1}=0, \;\; y_{i}:=x_{i}, \;\; i\geq 2.
	$$
	Then \((y_{i})_{i\in\N}\in H\) but \((y_{i})_{i\in\N}\notin R_n\) for every \(n\in\N\), contradicting the cover.
	Hence no such countable cover exists, and therefore by the definition of $\mu$, equation \eqref{JCS}, it follows that
	\[
	\mu(H)=\mu^*(H)=\inf\varnothing=\infty.
	\]
	Finally, we prove that \(H\) is an atom.
	Let \(B\subset H\) be measurable with \(\mu(B)<\infty\). Then, by definition of \(\mu\) there exists a cover \(B\subset\bigcup_{n\in\N} C_n\) with \(C_n\in\mathcal{F}(\mc{B},\l)\) and \(\sum_{n\in\N}\mu(C_n)<\infty\). Write
	$$C_n=\bigtimes_{i\in \N}C^{n}_{i}, \quad C^{n}_{i}\in \mc{B}.$$
	Intersecting with \(H\) we get
	\begin{align*}
	B\subset \bigcup_{n\in\N} (C_n\cap H) & = \bigcup_{n\in\N} \Big[\Big(\{0\}\times\bigtimes_{i=2}^{\infty} \R\Big)\cap \Big(\bigtimes_{i\in\N}C^{n}_{i}\Big)\Big] \\
	& \subset\bigcup_{n\in\N} \Big(\{0\}\times\bigtimes_{i=2}^{\infty} C_i^{n}\Big),
	\end{align*}
	and for each \(n\in\N\), 
	$$
	\mu(C_n\cap H)\leq\mu\left(\{0\}\times\bigtimes_{i=2}^{\infty} C_i^{n}\right)=\lim_{m\to\infty}\left[\lambda(\{0\})\prod_{i=2}^{m}\lambda(C_i^n)\right]=0.
	$$
	Hence, 
	$$
	\mu(B)\le\sum_{n\in\N} \mu(C_n\cap H)=0,$$
	so \(\mu(B)=0\).
	Thus every measurable \(B\subset H\) has \(\mu(B)\in\{0,\infty\}\), i.e., \(H\) is an atom. The same argument shows that $\{(x_{i})_{i\in\N}\in\R^{\N} : x_{1}=a\}$ is also an atom for $\mu$ for every $a\in\R$.
\end{proof}

The proof of Proposition \ref{HY} yields the following result.

\begin{proposition}
	\label{NF}
	The measure space \((\mathbb{R}^{\mathbb{N}},\mu)\) is not semi-finite.
\end{proposition}

\begin{proof}
		The hyperplane \(H\) defined in \eqref{H} satisfies \(\mu(H)=\infty\) while there is no \(E\in\mathcal{B}_{\infty}\) with \(E\subset H\) and \(0<\mu(E)<\infty\). Hence \(\mu\) is not semi-finite.
\end{proof}

Although \((\mathbb{R}^{\mathbb{N}},\mu)\) admits atoms, its restriction to any \(\sigma\)-finite measurable subset is purely nonatomic; in particular, \((\mathbb{R}^{\mathbb{N}},\mu)\) is relatively nonatomic. We begin with the following lemma.

\begin{lemma}
	\label{lem:mass-in-rect}
	If \(E\in\mathcal{B}_\infty\) satisfies \(0<\mu(E)<\infty\), then there exists \(R\in\mathcal{F}(\mc{B},\l)\) such that
	\(0<\mu(E\cap R)\le \mu(R)\).
\end{lemma}

\begin{proof}
	By the definition of $\mu$, equation \eqref{JCS}, there is a cover
	\(E\subset\bigcup_{n\in\N} R_n\) with \(R_n\in\mathcal{F}(\mc{B},\l)\) such that
	$$\sum_{n\in\N} \mu(R_n)<\mu(E)+1.$$
	Then
	\[
	\mu(E) = \mu\Big(\bigcup_{n\in\N}(E\cap R_n)\Big)\le \sum_{n\in\N} \mu(E\cap R_n).
	\]
	If \(\mu(E\cap R_n)=0\) for all \(n\in\N\), we get \(\mu(E)=0\), a contradiction. Hence there exists \(n_{0}\in\N\) such that \(\mu(E\cap R_{n_{0}})>0\).
\end{proof}

\begin{proposition}
	\label{prop:purely-nonatomic-finite}
	Let \(E\in\mathcal{B}_\infty\) with \(\mu(E)<\infty\).
	Then the restricted space \((E,\mu)\) is purely nonatomic.
\end{proposition}

\begin{proof}
	Let \(A\subset E\) be measurable with \(\mu(A)>0\).
	By Lemma~\ref{lem:mass-in-rect}, pick \(R\in\mathcal{F}(\mc{B},\l)\) with \(\mu(A\cap R)>0\).
	If \(\mu(A\cap R)<\mu(A)\), then \(B:=A\cap R\) satisfies \(0<\mu(B)<\mu(A)\), so \(A\) is not an atom.
	
	Otherwise, \(\mu(A\cap R)=\mu(A)\), hence \(\mu(A\setminus R)=0\). Denote \(R=\bigtimes_{i\in\N} C_i\). Then, since $(\R,\l)$ is purely nonatomic, there exists a positive measure Borel set $D\in\mc{B}$ such that \(D\subset C_1\) and \(\lambda(D)=\lambda(C_1)/2\). The rectangles 
	\[
	R^-:=D\times\bigtimes_{i=2}^{\infty} C_i,\qquad
	R^+:=(C_1\setminus D)\times\bigtimes_{i=2}^{\infty} C_i,
	\]
	satisfy \(R=R^-\sqcup R^+\) and \(\mu(R^-)=\mu(R^+)=\mu(R)/2\).
	Consequently,
	\[
	\mu(A)=\mu(A\cap R)=\mu(A\cap R^-)+\mu(A\cap R^+).
	\]
	If one summand lies in \((0,\mu(A))\), we are done.
	If not, we can suppose without lost of generality that \(\mu(A\cap R^-)=\mu(A)\) and \(\mu(A\cap R^+)=0\).
	Replace \(R\) by \(R^-\) and iterate.
	After \(n\in\N\) steps we get nested rectangles 
	$$
	\cdots \subset R_{n} \subset R_{n-1}\subset \cdots \subset R_{2} \subset R_{1}=R^{-}\subset R,
	$$
	such that \(\mu(R_{n})=2^{-n}\mu(R)\) and \(\mu(A\cap R_{n})=\mu(A)\).
	Since $\mu(A)>0$, there exists \(n_{0}\in\N\) such that \(2^{-n_{0}}\mu(R)<\mu(A)\), contradicting \(\mu(A)=\mu(A\cap R_{n_{0}})\le \mu(R_{n_{0}})\).
	Hence at some step we obtain \(B\subset A\) with \(0<\mu(B)<\mu(A)\).
	Thus no \(A\subset E\) with \(\mu(A)>0\) is an atom, i.e. \((E,\mu)\) is purely nonatomic.
\end{proof}

\begin{corollary}
	\label{cor:purely-nonatomic-sigma-finite}
	Let \(E\in\mathcal{B}_\infty\) be \(\sigma\)-finite for \(\mu\).
	Then \((E,\mu)\) is purely nonatomic.
\end{corollary}

\begin{proof}
	Write \(E=\bigcup_{n\in\N} E_n\) with \(E_n\in\mathcal{B}_\infty\) and \(\mu(E_n)<\infty\).
	If \(A\subset E\) is measurable with \(\mu(A)>0\), then \(\mu(A\cap E_n)>0\) for some \(n\in\N\). Consequently, $0<\mu(E_{n})<\infty$ and, by Proposition~\ref{prop:purely-nonatomic-finite}, \((E_n,\mu)\) is purely nonatomic.
	Then, there exists \(B\subset A\cap E_n\) with \(0<\mu(B)<\mu(A\cap E_n)\le \mu(A)\).
	Thus \(A\) is not an atom, and \((E,\mu)\) is purely nonatomic.
\end{proof}

The third result of this appendix shows that nontrivial continuous functions $f:\R^{\N}\to\R$ under the product topology of $\R^{\N}$ do not belong to $L^{p}(\mu)$ for every $1\leq p <\infty$. This is not the case for the usual Lebesgue measure on $\R^{n}$, $n\in\N$. We start by proving that every nonempty open subset of $\R^{\N}$ has infinite measure.
\begin{theorem}
	\label{A2}
	Every nonempty open subset $U\subset \R^{\N}$ satisfies $\mu(U)=\infty$.
\end{theorem}
\begin{proof}
	Let $U$ be a nonempty open subset of $\R^{\N}$ and take $x_{0}\in U$. As $\mu$ is translation invariant, we can suppose without lost of generality that $x_{0}=0$. Then, there exists $\varepsilon>0$ such that 
	$B_{\varepsilon}(0)\subset U$, where $B_{\varepsilon}(0)$ is the ball in $\R^{\N}$ with 
	center $0$ and radius $\varepsilon$. By the definition of the product topology, we deduce the existence of $n_{0}\in\N$ and $\d>0$ such that
	$$
	\mathcal{Y}_{n_{0}}:=\bigtimes_{n=1}^{n_{0}}(-\d,\d) \times \bigtimes_{n=n_{0}+1}^{\infty}\mathbb{R} \subset B_{\varepsilon}(0).
	$$
	This implies that
	$
	\mu(U)\geq  \mu(\mc{Y}_{n_{0}})$. Note that for each integer $k>n_0$,
	$$
	\mc{Y}_{k,n_{0}}:=\bigtimes_{n=1}^{n_{0}}(-\d,\d)\times \bigtimes_{n=n_{0}+1}^{\infty}[k,k+1)\subset \mc{Y}_{n_{0}}.
	$$
	Moreover, $\mc{Y}_{k,n_{0}}\in\mc{F}(\mc{B},\l)$ and $\mc{Y}_{k_{1},n_{0}}\cap \mc{Y}_{k_{2},n_{0}}=\varnothing$ for $k_{1}\neq k_{2}$. Consequently, by the $\s$-additivity of $\mu$, we deduce that	
	$$
	\mu(\mc{Y}_{n_{0}})\geq \mu\left( \bigsqcup_{k>n_{0}}\mc{Y}_{k,n_{0}}\right)=\sum_{k=n_{0}+1}^{\infty}\mu(\mc{Y}_{k,n_{0}})=\sum_{k=n_{0}+1}^{\infty}(2\d)^{n_{0}}=\infty.
	$$
	The proof is concluded.
\end{proof}

In the following, we denote by $\mc{C}(\R^{\N})$ the space of continuous functions $f:\R^{\N}\to\R$ under the product topology on $\R^{\N}$.

\begin{theorem}
	Let $f\in\mc{C}(\R^{\N})$ and $x_{0}\in\R^{\N}$ such that $f(x_{0})\neq 0$. Then, $f\notin L^{p}(\mu)$ for every $1\leq p < \infty$. In particular, it holds that $\mc{C}(\R^{\N})\cap L^{p}(\mu)=\{0\}$ for each $1\leq p < \infty$.
\end{theorem}
\begin{proof}
	Let $f\in \mc{C}(\R^{\N})$ and $x_{0}\in\R^{\N}$ be such that $f(x_{0})\neq 0$. As $\mu$ is translation invariant, we can suppose without lost of generality that $x_{0}=0$. Then, there exist $\varepsilon>0$ and $\rho>0$ such that 
	$B_{\varepsilon}(0)\subset\{|f|>\rho\}$.
	Since $f$ is continuous and we are dealing with the Borel $\s$-algebra of $\R^{\N}$, it is apparent that $f$ is measurable. Therefore, 
	$$
	\int_{\R^{\N}}|f|^{p} \; d\mu \geq \int_{B_{\varepsilon}(0)} |f|^{p} \; d\mu \geq \rho^{p} \mu(B_{\varepsilon}(0))=\infty,
	$$
	where the last equality follows from Theorem \ref{A2}. This concludes the proof.
	\end{proof}

\end{document}